\documentclass[11pt]{amsart}

\usepackage{enumitem}
\usepackage{verbatim}
\usepackage[normalem]{ulem}
\usepackage{graphicx}
\usepackage{xcolor}
\usepackage{hyperref}

\newcommand{\foot}[1]{\mbox{}\marginpar{\raggedleft\hspace{0pt}\tiny #1}}
\newcommand{\vc}[1]{\foot{VC: #1}}

\newcommand{\ph}{\varphi}
\newcommand{\eps}{\epsilon}
\newcommand{\ulim}{\varlimsup}

\newcommand{\NN}{\mathbb{N}}
\newcommand{\RR}{\mathbb{R}}
\newcommand{\JJ}{\mathbf{J}}
\newcommand{\II}{\mathbf{I}}

\newcommand{\MMM}{\mathcal{M}}

\newcommand{\HHH}{\mathcal{H}}
\newcommand{\GGG}{\mathcal{G}}

\newcommand{\di}{\partial}
\newcommand{\wM}{\widetilde{M}}
\newcommand{\ideal}{\di\widetilde M}
\newcommand{\hexp}{h_{\mathrm{exp}}^\perp}
\newcommand{\htop}{h_{\mathrm{top}}}
\newcommand{\hvol}{h_{\mathrm{vol}}}
\newcommand{\dididi}{ (\partial \wM \times \partial \wM) \setminus \diag}
\newcommand{\sqbd}{\di^2\wM}
\newcommand{\tmu}{\widetilde{\mu}}

\DeclareMathOperator{\grad}{grad}

\DeclareMathOperator{\card}{card}
\DeclareMathOperator{\diam}{diam}
\DeclareMathOperator{\Leb}{Leb}

\DeclareMathOperator{\cl}{cl}
\DeclareMathOperator{\vol}{vol}
\DeclareMathOperator{\supp}{supp}
\DeclareMathOperator{\diag}{diag}
\DeclareMathOperator{\pr}{pr}

\newtheorem{theorem}{Theorem}[section]
\newtheorem{lemma}[theorem]{Lemma}
\newtheorem{proposition}[theorem]{Proposition}
\newtheorem{corollary}[theorem]{Corollary}
\newtheorem{definition}[theorem]{Definition}

\newtheorem*{thma*}{Theorem}

\theoremstyle{remark}
\newtheorem{remark}[theorem]{Remark}

\DeclareMathOperator{\NE}{NE}

\numberwithin{equation}{section}
\numberwithin{figure}{section}

\begin{document}

\title[Geodesic flows without conjugate points]{Uniqueness of the measure of maximal entropy for
geodesic flows on certain manifolds without conjugate points}
\author{Vaughn Climenhaga}
\author{Gerhard Knieper}
\author{Khadim War}
\address{Dept.\ of Mathematics, University of Houston, Houston, TX 77204}
\address{Dept.\ of Mathematics, Ruhr University Bochum, 44780 Bochum, Germany}
\address{IMPA, Estrada Dona Castorina 110, Rio de Janeiro, Brazil}

\email{climenha@math.uh.edu}
\email{gerhard.knieper@rub.de}
\email{khadim@impa.br}
\date{\today}
\subjclass{ 37D25, 37D35, 37D40, 53C22}
\keywords{Geodesic flows without conjugate points, Measure of maximal entropy}

\begin{abstract}
We prove that for closed surfaces $M$ with Riemannian metrics without conjugate points and genus $\geq 2$ the geodesic flow on the unit tangent bundle $T^1M$ has a unique measure of maximal entropy.
Furthermore, this measure is fully supported on $T^1M$, 
is the limiting distribution of closed orbits,  and the flow is mixing with respect to this measure.  
We formulate conditions under which this result extends to higher dimensions.
\end{abstract}

\thanks{V.C.\ is partially supported by NSF grant DMS-1554794.}
\thanks{G.K. and K.W. are partially supported by the German Research Foundation (DFG),
CRC/TRR 191, \textit{Symplectic structures in geometry, algebra and dynamics.}}
\maketitle

\section{Introduction}

Let $X$ be a compact metric space and $F = \{f_t\colon X\to X\}_{t\in \RR}$ a continuous flow. 
The complexity of the flow can be quantified by the \emph{topological entropy}
\[
\htop(F) = \lim_{\eps\to 0} \limsup_{t\to\infty} \frac 1t\log \Lambda_t(X,F,\eps),
\]
where $\Lambda_t(X,F,\eps)$ is the maximum cardinality of a \emph{$(t,\eps)$-separated set} -- that is, a set $E\subset X$ such that for every $x,y\in E$, there is $s\in [0,t]$ such that $d(f_s x, f_s y) \geq \eps$.
The \emph{variational principle}  \cite[Theorem 8.6]{pW82} says that
\[
\htop(F) = \sup \{ h_\mu(f_1) : \mu \in \MMM_F(X) \},
\]
where $\MMM_F(X)$ is the space of flow-invariant Borel probability measures on $X$, and $h_\mu(f_1)$ is measure-theoretic entropy.  A measure $\mu \in \MMM_F(X)$ that achieves the supremum is called a \emph{measure of maximal entropy (MME)}.

If $M$ is a smooth closed Riemannian manifold of negative sectional curvature, then by classical work of Bowen and Margulis, the geodesic flow $F = \{f_t\}_{t\in \RR}$ on the unit tangent bundle $T^1M$ has a unique measure of maximal entropy. Moreover, the unique MME is fully supported  on $T^1M$ and is the limiting distribution of closed orbits of the geodesic flow \cite{rB73,rB74}.  This proof uses Markov partitions; the result can also be proved using expansivity and the specification property \cite{rB75,eF77}.

The corresponding result for rank 1 manifolds with nonpositive sectional curvatures, where the tools in the previous paragraph may all fail, was proved by the second author in \cite{gK98} using Patterson--Sullivan measures.  An alternate proof was recently given in \cite{BCFT} using non-uniform versions of expansivity and specification introduced in \cite{CT16}.

Once $M$ is allowed to have some positive sectional curvatures, there are two natural conditions to impose that still guarantee some (non-uniform) hyperbolicity for the geodesic flow. The more restrictive of these is \emph{no focal points}; in this setting many of the tools from nonpositive curvature still hold, and it is possible to adapt the approaches described above; see \cite{GR,LWW,CKP}.
  
We work in the broader setting of \emph{no conjugate points}, where most of the tools from nonpositive curvature fail in general; see \cite{BBB87,kB92} for some discussion of the phenomena that can occur.  In particular, each of the above approaches encounters substantial difficulties, so that there is no straightforward generalization of either \cite{gK98} or \cite{BCFT}.  Nevertheless, by using tools from coarse geometry together with the result from \cite{CT16} on non-uniform expansivity and specification, we obtain the following result.

\begin{theorem}\label{thm:surfaces}
Let $M$ be a smooth closed surface of genus $\geq 2$, and let $g$ be a Riemannian metric on $M$ without conjugate points.  Then the associated 
geodesic flow on $T^1M$ has a unique measure $\mu$ of maximal entropy.  The measure $\mu$  has full support, is Bernoulli (hence mixing), 
and is the limiting distribution of (homotopy classes of) closed geodesics in the sense of Definition \ref{def:periodics}. In particular, the set of closed orbits is dense in $T^1M$.
\end{theorem}

In fact, Theorem \ref{thm:surfaces} is a specific case of a more general result (Theorem \ref{thm:higher-dim}), 
which establishes  uniqueness for a broad class $\mathcal{H}$ of Riemannian manifolds with no conjugate points in arbitrary dimension.  The key properties of a manifold $(M,g)\in \mathcal{H}$ are:
\begin{itemize}
\item $M$ admits a negatively curved ``background'' Riemannian metric; 
\item geodesics emanating from a common point on the universal covering eventually diverge;
\item the fundamental group is residually finite;
\item all invariant measures of ``nearly maximal'' entropy must have support on the expansive set, see  \eqref{eqn:E}.
\end{itemize}
See \S\ref{sec:high} for precise definitions, an explanation of why $\mathcal{H}$ contains every surface of genus $\geq 2$, and a discussion of how restrictive these conditions are.

Uniqueness and ergodicity of the MME are provided by an application of \cite{CT16} (though see Remark \ref{rmk:coarse} below for a discussion of the stark differences between our result here and the one in \cite{BCFT}, and the novelties required in the present setting).
Bernoullicity in Theorem \ref{thm:surfaces} requires a result by Ledrappier, Lima, and Sarig \cite{LLS16}, which only applies when $\dim M=2$.  In higher dimensions, we can still obtain mixing by proving that the Patterson--Sullivan construction gives the unique MME (once uniqueness is known).
 For nonpositively curved closed manifolds of rank 1, Patterson--Sullivan measures and their geometric applications were studied in  \cite{gK97}. The properties of the Patterson--Sullivan construction were then used by Babillot \cite{mB02} to prove mixing of the flow with respect to the unique measure of maximal entropy. 
For manifolds in the class $\mathcal{H}$, we use the uniqueness result to demonstrate that the construction in \cite{gK98} gives the MME, and thus
obtain mixing as in \cite[Theorem 2]{mB02}.

\begin{theorem}\label{thm:PS}\label{thm:higher-dim}
Let $M$ be a 
Riemannian manifold with no conjugate points that lies in the class $\mathcal{H}$.
Then the associated geodesic flow on $T^1M$ has a unique measure of maximal entropy $\mu$.  The measure $\mu$ has full support,
is mixing, and is the limiting distribution of (homotopy classes of) closed geodesics in the sense of Definition \ref{def:periodics}. In particular, the set of closed orbits is dense in $T^1M$.
\end{theorem}

\begin{remark}
All closed surfaces of genus at least 2 are in the class  $\mathcal{H}$.
On the other hand we do not know of  any examples of  closed manifolds without conjugate points that admit a metric of negative curvature 
which are not contained in $\mathcal{H}$; see \S\ref{sec:high} for more details. Thus Theorem \ref{thm:higher-dim} gives uniqueness and mixing for the MME on all the known examples of manifolds without conjugate points supporting a metric of negative curvature.
\end{remark}

\begin{remark}
Our proof that the Patterson--Sullivan construction gives a 
measure of maximal entropy
only uses that  $(M,g)$ is a  Riemannian manifold without conjugate points 
satisfying the first two properties in the definition of $\mathcal{H}$ (see \S\ref{sec:PS}).
Under the extra (strong) assumption of expansivity, it was proved by Aur\'elien Bosch\'e in his thesis \cite{aB18} that this is in fact the unique MME, but in our more general setting the Patterson--Sullivan approach does not provide a proof of uniqueness.
\end{remark}

\begin{remark}\label{rmk:coarse}
The proof of uniqueness in Theorem \ref{thm:higher-dim} uses nonuniform expansivity and specification properties as in \cite{CT16}.  The idea is to leverage the coarse hyperbolicity provided by the Morse Lemma, which gives a shadowing property in the universal cover at an \emph{a priori} very large scale $R$.  This leads to a specification property, which is enough for uniqueness provided the ``obstructions to expansivity'' are controlled at an even larger scale.  For this we use residual finiteness of $\pi_1(M)$ to pass to a finite cover of $M$ whose injectivity radius is much larger than $R$.  This application of tools from coarse geometry has no analogue in prior work using non-uniform expansivity and specification properties \cite{BCFT,CKP}; in those settings enough hyperbolicity is available to establish specification at arbitrarily small scales for a natural family of orbit segments, but those arguments do not work here due to the failure of various nice properties such as monotonicity of Jacobi fields and continuity of the stable and unstable foliations.
\end{remark}

The proof of mixing in Theorem \ref{thm:PS} goes along the same lines as in \cite[Theorem 2]{mB02} for rank 1 manifolds where there is an abstract result for mixing provided
 that the length spectrum is not \emph{arithmetic}, i.e., is not a discrete subgroup of \(\mathbb{R}\). However, in our case, there are some technicalities 
 related to the fact that the flow is not expansive and the expansive set is not open in general. More precisely, in \cite{mB02},  the continuity of the cross-ratio function 
 that implies the non-arithmeticity of the length spectrum is established by considering its restriction to the expansive set and using uniform hyperbolicity for the recurrent subset \cite[Proposition 4.1]{gK98}. 
See \S\ref{sec:asymptotic}, and especially Lemma \ref{lem:tnk}, for the results that play an analogous role in our setting.
We remark that the cross-ratio has an  analogue in the context of contact Anosov flows, the so called temporal function \cite[Figure~2]{cL04} whose \(C^2\) regularity was  key in the  proof of exponential mixing for contact Anosov flows.

\begin{remark}
In negative curvature, Margulis proved an asymptotic formula for the number of closed geodesics \cite{gM69,gM04}, relying heavily on 
a leafwise description of the measure of maximal entropy that can be also be interpreted via the Patterson--Sullivan approach \cite{vK90}. A key role is played by the fact that periodic orbits equidistribute to the MME.
%for every $\delta>0$, the MME is the limit of the measures $\mu_T$ equidistributed over periodic orbits with length in $[T-\delta,T]$.
%the MME satisfies an equidistribution result for periodic orbits.

It was natural to conjecture \cite{BK85} that a similar result holds in nonpositive curvature, where one must consider free homotopy classes of closed geodesics. The construction of the unique MME by the second author \cite{gK98} resolved part of this conjecture, and the Margulis asymptotic itself has recently been announced in a preprint of Ricks \cite{rR}.

Theorem \ref{thm:PS} can be considered as a starting point for extending the Margulis asymptotic to manifolds without conjugate points, and we complete this process in a separate paper \cite{CKW}, which in particular includes an equidistribution result
that strengthens the one here; see Remark \ref{rmk:CKW}.
%for the periodic orbit measures $\mu_T$ \cite[Theorem 1.3]{CKW}. A weaker equidistribution result -- in which $[T-\delta,T]$ is replaced by $[0,T]$ -- could be established relatively quickly from Theorem \ref{thm:PS} using an estimate from \cite{CK02} on the number of free homotopy classes containing a closed geodesic with length in $[0,T]$, 
%%the fact that if $P(T)$ denotes the number of free homotopy classes containing a closed geodesic of period less than $T$, then there are constants $A,B$ such that $AT^{-1} e^{T\htop(F)} \leq P(T) \leq Be^{T\htop(F)}$ for all sufficiently large $T$, which was proved in 
%together with an argument from ergodic theory estimating the entropy of any limit point of measures constructed from $(T,\epsilon)$-separated sets. There are some subtleties in relating free homotopy classes and $(T,\epsilon)$-separated sets, though, and to avoid getting bogged down in these details, which can be found in \cite[\S2.4 and \S6.1]{CKW}, we omit this equidistribution result from the formal statements of Theorems \ref{thm:surfaces} and \ref{thm:PS}.
\end{remark}

\subsection*{Structure of the paper}

In \S\ref{sec:geometry}, we give definitions and properties of manifolds without conjugate points and state the Morse lemma in Theorem \ref{thm:Morse}. In  \S\ref{sec:es}, we give the definition of specification and state the general results for uniqueness in Theorem \ref{thm:general}. In \S\ref{sec:per-limit}, we describe the property of equidistribution of closed geodesics.
In \S\ref{sec:algebra}, we discuss the property of fundamental group being residually finite which is enough to have a finite cover of arbitrarily large injectivity radius.
 
In \S\ref{sec:high}, we give a precise definition of the class $\mathcal{H}$ of manifolds to which Theorem \ref{thm:higher-dim} applies, and we prove Theorem \ref{thm:surfaces} under Theorem \ref{thm:higher-dim}.  \S\ref{sec:proof} is devoted to the proof of Theorem \ref{thm:higher-dim}, modulo the mixing property.
In \S\ref{sec:PS}, we prove that the MME in Theorem \ref{thm:higher-dim} is given by a Patterson--Sullivan construction as in \cite{gK98}. This construction is then used in \S\ref{sec:mixing} to prove that the flow is mixing with respect to the measure of maximal entropy which completes the proof of Theorem \ref{thm:PS}.  Certain technical proofs are given in the appendices.
 
\subsection*{Acknowledgments}

We are grateful to the anonymous referees for a careful reading and for many helpful suggestions.

\section{Background}\label{sec:background}

\subsection{Geometry of manifolds without conjugate points}\label{sec:geometry}

%\subsection{Geodesic flows and coarse hyperbolicity}%\label{sec:geodesic}

\subsubsection{Geodesic flows}\label{sec:geodesic}

Given a smooth closed $n$-dimensional Riemannian manifold $(M,g)$,
we write $F=\{f_{t}: T^1M\to T^1M \}_{t \in \RR} $ for the geodesic flow on the unit tangent bundle defined by $f_t(v) = \dot{c}_v(t)$, where $c_v$ is the unique geodesic on $M$ with $\dot{c}_v(0) = v$.  It is convenient for us to use the metric
$
d_1(v,w) = \max_{t \in [0,1]} d(c_v(t), c_w(t))
$
on $T^1M$ where  $d$ is the metric on $M$ induced by the Riemannian metric.

Throughout the paper, we will assume that $(M,g)$ is a smooth closed Riemannian manifold 
without conjugate points. Such manifolds are characterized
by the fact that the exponential map $\exp_p\colon T_pM \to M$ is not singular for all $p \in M$ or equivalently,  each nontrivial orthogonal Jacobi field vanishes at most at one point. 
The following relationships between this property and other conditions that give some kind of hyperbolic behavior are straightforward:
\begin{center}
\emph{nonpositive sectional curvature %$\stackrel{\Rightarrow}{\not\Leftarrow}$ 
$\Rightarrow$
no focal points
$\Rightarrow$
no conjugate points.}
\end{center}
The converse implications all fail in general.
%It is easy to see that nonpositive sectional curvature or more generally no focal points implies no conjugate points.

The  Cartan--Hadamard Theorem 
says that  the universal cover \(\wM\) of $(M,g)$ is diffeomorphic to $\RR^n$  via the exponential map and  therefore for every pair of distinct points $p,q\in \wM$, there is a unique geodesic segment $c\colon [a,b]\to \wM$ such that $c(a)=p$ and $c(b) = q$: geodesics are globally minimizing in \(\wM\).

The  group of deck transformations $\Gamma$ is isomorphic to the fundamental group $\pi_1(M)$ and acts isometrically on $\wM$.  In particular, $M$ is isometric to the quotient $\wM/\Gamma$.
We write $\pr \colon \wM \to M$ for the canonical projection and $\pr_* \colon T^1\wM \to T^1M$ for the map it induces between the unit tangent bundles.  Given a finite flow-invariant measure $\mu$ on $T^1M$, the \emph{lift} of $\mu$ is the $\sigma$-finite flow-invariant and $\Gamma$-invariant measure $\tmu$ on $T^1\wM$ defined by
\begin{equation}\label{eqn:tmu-lift}
\tmu(A) = \int_{T^1M} \#(\pr_*^{-1}(v) \cap A) \,d\mu(v).
\end{equation}

We let \(\pi\colon T^1\wM\to \wM\) be the standard projection.

\subsubsection{Coarse hyperbolicity using a background metric}

An important strategy throughout the paper will be to compare the geometric properties of $M$ with respect to two different Riemannian metrics $g$ and $g_0$, where $g$ is the original metric we are given, and $g_0$ is a `background' metric which we will always assume to have negative sectional curvatures.  

\begin{remark}\label{rmk:g0}
Existence of a negatively curved background metric places genuine topological restrictions on $M$.  In particular, as stated above the universal covering
is diffeomorphic to $\RR^n$ and by Preissmann's theorem each abelian subgroup of the fundamental group is infinite cyclic. Therefore, in dimension two it forces the genus of $M$ to be at least $2$ and in higher dimensions excludes examples such as Gromov's graph manifolds of nonpositive curvature for which uniqueness of the MME is known  \cite[\S6]{gK98}. On the other hand in this paper we do not impose
any local assumptions such as restrictions on the curvature.
\end{remark}

We will write $d,d^0$ for the distance functions associated to $g,g_0$ on both $M$ and $\wM$.  The first crucial observation is that by compactness of $M$ and the equivalence of  the quadratic forms of \(g\) and \(g_0\),  there exists a constant \(A>0\) such that for every \(p,q\in\wM\), we have
\begin{equation}\label{lem:unif-equiv}
A^{-1}\cdot d^{0}(p,q)\leq d(p,q)\leq A\cdot d^{0}(p,q).
\end{equation}
This has an important consequence for topological entropy.  Generalizing an earlier result of Manning \cite{aM79} in nonpositive curvature, it was shown by Freire and Ma\~n\'e \cite{FM82} that on a closed manifold without conjugate points, 
the volume growth in the universal cover is equal to the topological entropy of the geodesic flow: 
\begin{equation}\label{eq:topent}
\htop(F)=\lim_{r\to\infty}\frac{1}{r}\log \vol (B(p,r)),
\end{equation}
where \(p\) is any point in \(\wM\) and \(B(p,r) \subset \wM\) is the ball centered at \(p\) of radius \(r\). 
We say that $g$ has positive topological entropy if its geodesic flow $F$ has $\htop(F)>0$.
The following is an immediate consequence of \eqref{lem:unif-equiv} and \eqref{eq:topent}.

\begin{lemma}\label{lem:h>0}
If $M$ admits a metric $g_0$ without conjugate points that has positive topological entropy, then \emph{every} metric $g$ without conjugate points on $M$ has positive topological entropy.
In particular, this occurs if $M$ admits a metric $g_0$ with negative sectional curvatures.
\end{lemma}

From now on we consider a closed Riemannian manifold $(M,g)$ without conjugate points which admits a background metric $g_0$ of negative curvature, and thus has positive topological entropy.
%Let \(d, d^0\) denote the distance functions associated to the lift of the metrics \(g,g_0\) on \(\wM\). 
This lets us deduce certain coarse hyperbolicity properties, for which we  recall that Hausdorff distance \(d_H\) between two subsets $C_1,C_2\subset \wM$ (with respect to \(g\))  is defined by 
\[
d_H(C_1, C_2) :=\inf\{r>0: C_1\subset T_r(C_2), C_2\subset T_r(C_1)\}
\]
where \(T_r(C):=\{p\in\wM: d(p, C)\leq r\}\). We denote by \(d_H^0\) the Hausdorff distance with respect to \(d^0\).
The  following result  goes back to Morse \cite{mM24} in dimension two, and Klingenberg \cite{wK71} in higher dimensions.  We follow the statement from \cite[Theorem 3.3]{GKOS14}; see \cite[Lemma 2.7]{gK02} for a detailed proof.

\begin{theorem}[Morse Lemma]\label{thm:Morse}
If $g,g_0$ are two metrics on $M$ such that $g$ has no conjugate points and $g_0$ has negative curvature, then there is a constant $R_0 = R_0(g,g_0) > 0$ such that if $c\colon [a,b]\to \wM$ and $\alpha\colon [a_0,b_0] \to \wM$ are minimizing geodesic segments with respect to $g,g_0$, respectively, joining $c(a)=\alpha(a_0)$ to $c(b)=\alpha(b_0)$, then $d_H(c[a,b],\alpha[a_0,b_0]) \leq R_0$.
\end{theorem}

We prove the following consequence in Appendix \ref{sec:Morse}.

\begin{lemma}\label{lem:endpts-suffice}
Let $g$ be a metric on $M$ without conjugate points. If \(M\)  admits a metric of negative curvature then for every $R_1> 0$ there is $R_2>0$ such that for every \(T>0\) if $c_1,c_2\colon [0,T]\to \wM$ are two geodesics with $d(c_1(0),c_2(0)) \leq R_1$ and $d(c_1(T),c_2(T)) \leq R_1$, then $d(c_1(t),c_2(t)) \leq R_2$ for all $t\in [0,T]$.
\end{lemma}

\begin{remark}
When $g$ itself has nonpositive curvature, Lemma \ref{lem:endpts-suffice} follows easily from the convexity of the distance function $t\mapsto d(c_1(t), c_2(t))$ between two geodesics. In our more general setting we rely on the background metric of negative curvature and use the Morse Lemma.
\end{remark}

\subsubsection{Busemann functions and horospheres}

Given \(v\in T^1\wM\), recall that $c_v\colon \RR \to \wM$ denotes the geodesic with $\dot c_v (0) = v$.  For each $t>0$, consider the function on $T^1\wM$ defined by
$b_{v,t}(p) := d(p,c_v(t)) - t$.
\begin{lemma}[{\cite[Proposition 1]{jE77}}]\label{lem:busemann}
For every $v\in T^1\wM$ and $p\in \wM$, the limit $b_v(p) := \lim_{t\to\infty} b_{v,t}(p)$ exists and defines a $C^1$ function on $\wM$.  Moreover, $\grad b_v(p) = \lim_{t\to \infty} \grad b_{v,t}(p)$.
\end{lemma}

Existence of the limit is essentially due to the fact that geodesics on $\wM$ are globally minimizing.  The limiting function $b_v$ is called a \emph{Busemann function}, and was shown in \cite{gK85} to be in fact $C^{(1,1)}$.

Observe that if $t\geq \tau$, then $b_{v,t}(c_v(\tau)) = d(c_v(\tau),c_v(t)) - t = (t-\tau)-t=-\tau$, so we have
\begin{equation}\label{eqn:grad-bv}
b_v(c_v(t)) = -t
\text{ and }
-\grad b_v(c_v(t)) = \dot{c}_v(t) = f_t(v).
\end{equation}

Given \(v\in T^1\wM\), the \emph{stable and unstable horospheres} \(H^s(v)\) and \(H^u(v)\) are the subsets of $\wM$ defined by 
\begin{equation}\label{eqn:Hsu}
H^s(v) :=\{p\in\wM: b_v(p)=0\}
\quad\text{and}\quad
%H^u(v) :=\{x\in\wM: b_{-v}(x)=0\}.
H^u(v) := H^s(-v).
\end{equation}

We refer to $c_v(\infty)$ as the \emph{center} of the stable horosphere $H^s(v)$.  Similarly, $c_v(-\infty) := c_v^-(\infty)$ is the center of $H^u(v)$, where we write $c_v^-(t) = c_v(-t)$.
%{\color{blue}
%The Busemann flow
%\(\psi_t^{\xi}\) is defined to be the flow on \(\wM\) tangent to \(\grad b_v\) where \(v\) is such that \(c_v(\infty)=\xi\). From \eqref{eqn:grad-bv}, we get the invariance of the horospheres by the Busemann flow:
%\begin{equation}\label{eq:invhor}
%\psi^\xi_t(H^s(v))=H^s(f_tv).%\quad\text{ for }\quad \sigma=s,u.
%\end{equation}
%}

We also consider the \emph{(weak) stable and unstable manifolds}, which are the subsets of $T^1\wM$ defined by
\begin{equation}\label{eqn:Wsu}
W^s(v):=\{-\grad b_v(p) \mid p \in\wM\} 
\quad\text{and}\quad
 W^u(v) := - W^s(-v).
\end{equation}
Some justification for this terminology will be given in the next section.
For the moment we observe that $W^s(v)$ is the union of the  unit normal vector fields to horospheres centered at $c_v(\infty)$, and that it is $F$-invariant by \eqref{eqn:grad-bv}.
The regularity of the Busemann function implies that the horospheres are \(C^{(1,1)}\) manifolds and the  stable and unstable manifolds are Lipschitz.

\subsubsection{Manifolds of hyperbolic type and the boundary at infinity}

We say that  $(M,g)$   has the \emph{divergence property} if any pair of geodesics  $c_1 \neq c_2$ in $( \wM, g)$ with $c_1(0) =c_2(0)$ diverge, i.e.,
\begin{equation}\label{eqn:divergence}
\lim\limits_{t \to \infty} d(c_1(t), c_2(t)) = \infty.
\end{equation}

\begin{remark}\label{rmk:divergence}
Every surface without conjugate points has the divergence property \cite{lwG56}. In higher dimensions it is unknown whether this condition always holds.
\end{remark}

The following definition and theorem are due to Eberlein \cite{pEb72}.

\begin{definition}
A simply connected Riemannian manifold $\wM$ without conjugate points 
is a \emph{(uniform) visibility manifold} if for every $\epsilon>0$ there exists $L>0$ such that whenever a geodesic $c\colon [a, b] \to\wM$ stays at distance at least $L$ from some point $p\in\wM$, then the angle sustained by $c$ at $p$ is less than $\epsilon$, that is
  \begin{equation*}
    \angle_p(c)=\sup_{a\leq s,t\leq b} \angle_p(c(s),c(t))<\epsilon.
  \end{equation*}
\end{definition}

\begin{theorem}\label{thm:visibility}
Let $(M,g)$ be a closed manifold without conjugate points which admits a background metric $g_0$ of negative curvature.
Then $( \wM, g)$ is a visibility manifold if and only if $(M,g)$ has the divergence property.
\end{theorem}

\begin{remark}
We remark that Ruggiero  \cite{rRu} obtained this result without assuming the existence of a negatively curved background metric.
Instead he assumed the weaker condition that $ \wM$ is hyperbolic in the sense of Gromov \cite{mG87}. 
\end{remark}

\begin{definition}\label{def:hyp-type}
We say that a closed manifold $(M,g)$ without conjugate points is of \emph{hyperbolic type} provided it carries a background metric $g_0$ of negative curvature
and it satisfies the divergence property.
\end{definition}

\begin{remark}\label{rmk:higher-genus}
By Remark \ref{rmk:divergence} and the classification of surfaces, every surface of genus $\geq2$ without conjugate points is of hyperbolic type.
\end{remark}

Now we assume that $(M,g)$ is of hyperbolic type, and describe a compactification of $\wM$ following Eberlein \cite{pEb72}.  Two geodesic rays  $c_1, c_2\colon [0,\infty) \to \wM$ are called \emph{asymptotic}
if $d(c_1(t), c_2(t))$ is bounded for $t \ge 0$.  This is an equivalence relation; we denote by  \(\ideal\) the set of equivalence classes and call its elements \emph{points at infinity}. We denote the equivalence class of a geodesic ray (or geodesic) $c$  by $c(\infty)$.
The following construction is useful.

\begin{lemma}\label{lem:ct}
Given $p\in \wM$ and $v\in T^1\wM$, for each $t>0$ let $c_t$ be the geodesic from $p$ to $c_v(t)$, with $c_t(0) = p$.
Then the limit $w := \lim_{t\to\infty} \dot{c}_t(0) \in T_p^1 \wM $ exists and has the property that $w = -\grad b_v(p)$ and $c_w(\infty) = c_v(\infty)$.
\end{lemma}
\begin{proof}
Lemma \ref{lem:busemann} gives existence of the limit and the claim regarding $b_v$.  To show that $c_w(\infty) = c_v(\infty)$, we apply Lemma \ref{lem:endpts-suffice} with $R_1 = d(p,\pi(v))$
to the geodesics $c_t,c_v \colon [0,t]\to \wM$, which gives 
$d(c_t(s), c_v(s)) \leq R_2$ for all $s\in [0,t]$ where $R_2$ depends only on $p$ and $v$.  Since $c_w(s) = \lim_{t\to\infty} c_t(s)$ for all $s>0$, we conclude that  $d(c_w(s),c_v(s)) \leq R_2$, and thus $c_w(\infty) = c_v(\infty)$.
 \end{proof}

\begin{lemma}\label{lem:p-to-xi}
Given any $p\in \wM$ and $\xi\in \ideal$, there is a unique geodesic ray $c\colon [0,\infty)\to \wM$ with $c(0) = p$ and $c(\infty) = \xi$.
Equivalently, the map  $f_p\colon T_p^1\wM \to \ideal$ defined by $f_p(v) = c_v(\infty)$ is a bijection.
\end{lemma}
\begin{proof}
Surjectivity follows from Lemma \ref{lem:ct}.  Injectivity is an immediate consequence of the divergence property.
\end{proof}

Following Eberlein we equip $\ideal$ with a topology that makes it a compact metric space homeomorphic to $S^{n-1}$.  Fix $p\in \wM$ and let $f_p\colon T^1_p\wM \to \partial \wM$ be the bijection $v \mapsto c_v(\infty)$ from Lemma \ref{lem:p-to-xi}. The topology (sphere-topology) on $\partial \wM$ is defined such that $f_p$
becomes a homeomorphism.
Since for all $q \in \wM$ the map $f_q^{-1} f_p \colon T^1_p\wM \to T^1_q\wM$ is a homeomorphism, see \cite{pEb72},
the topology is independent on the reference point $p$.

The topologies on $\partial \wM$ and $\wM$
extend naturally
 to  $\cl (\wM): =  \wM\cup \partial \wM$
by requiring that the map
$\varphi\colon B_1(p) = \{v \in T_p \wM:   \|v\| \le 1\} \to \cl(\wM)$
defined by
\[
\varphi(v) = \begin{cases}
 \exp_p\left(\frac{v}{1-\|v\|}\right) & \|v\| < 1\\
f_p(v) & \|v\| = 1
\end{cases}
\]
is a homeomorphism. This topology, called the cone topology, was introduced by Eberlein
and O'Neill \cite{EO73} in the case of Hadamard manifolds and by Eberlein   \cite{pEb72} in the case of visibility manifolds. In
particular, $\cl( \wM) $ is homeomorphic to a closed ball in
$\mathbb{R}^n$.
The relative topology on $\ideal$ coincides with the sphere topology, and the relative topology on $\wM$ coincides with the manifold topology.
% The relative topology on $\partial  \wM$, respectively, on $ \wM$ coincides with the sphere topology, respectively, the topology on the manifold $ \wM$.
\begin{remark}\label{rem:minmal}
The isometric action of $\Gamma=\pi_1(M)$ on $\wM$ extends to a continuous action on $\ideal$.
Since by \cite{pEb72} the geodesic flow is topologically transitive, every $\Gamma$-orbit in $\ideal$ is dense, i.e. the action on  $\ideal$ is minimal.
\end{remark}

\begin{definition}\label{def:busemann}
Given $p\in \wM$ and $\xi\in\ideal$, let $v\in T_p^1\wM$ be the unique unit tangent vector at $p$ such that $c_v(\infty) = \xi$.  
We call $b_p(q,\xi) := b_v(q)$ the Busemann function based at $\xi$ 
and normalized by $p$, i.e. $b_p(p,\xi) =0$.
\end{definition}

The following important property of visibility manifolds is due to Eberlein \cite[Proposition 1.14]{pEb72}.
\begin{proposition}\label{prop:Eb1}
Let $\wM$ be a visibility manifold and 
$$
A = \{(q,z) \in \wM \times \cl(\wM) \mid q \not= z \}. 
$$ 
Consider the map $V\colon A \to T^1\wM$ such that
$V(q,z) \in  T^1_q\wM$ is the unique vector with $c_{V(q,z)}(d(q,z)) = z$. Then $V\colon A \to T^1\wM$ is continuous with respect to the topology defined above.
\end{proposition}
\begin{corollary}\label{cor:c1}
For  $p,q,z \in \wM$ define $b_p(q,z) = d(q,z)-d(p,z)$. Then for all $p \in \wM $, compact subsets $K \subset \wM$, $\xi \in \partial  \wM$ and all $\epsilon >0$
there exists an open set $U \subset  \cl(\wM)$ such that
\[
| b_p(q,z) - b_p(q,\xi)| < \epsilon
\]
for all $q \in K$ and $z \in U$. In particular, we have
\begin{equation}\label{eqn:bpqxi}
\lim\limits_{z \to \xi }d(q,z) -d(p,z) = b_p(q, \xi)
\text{ for every $p,q\in \wM$ and $\xi\in\ideal$}.
\end{equation}
\end{corollary}
\begin{proof}
For $p \in \wM$ and $z \in \cl(\wM)$ consider the function
$q \mapsto b_p(q,z) $. Then for $q \not= z$ we have $\grad b_p(q,z) = - V(q,z)$. For a compact set
$K \subset \wM$ define 
$$
B(K) = \cl \{q \in c[0,1] \mid c: [0,1] \to \wM \; \ \text{geodesic with} \; \ c(0) = p, c(1) \in K \}
$$
which is compact as well. Choose $r>0$ such that $K \subset B(p,r)$ and $\epsilon>0$. Since $V\colon A \to  T^1\wM$ is continuous it is uniformly continuous on compact subsets.
In particular, for $\xi \in  \partial  \wM$ there exists
a neighborhood $U \subset  \cl(\wM)$ such that $U \cap B(K) = \emptyset$ and
$
\|V(q, z) -V(q, \xi) \| < \frac{\epsilon}{r} 
$
for all $q \in B(K)$ and $z\in U$. For $q \in K$ and $z \in U$ and the geodesic $ c:[0,1] \to \wM $  with $c(0) =p $ and $c(1) =q $. we obtain
\begin{multline*}
| b_p(q, z) - b_p(q, \xi) | =  \left |\int_0^1 \frac{d}{dt}(b_p(c(t), z) - b_p((c(t), \xi)) dt \right |\\
= \left | \int_0^1 \langle \grad b_p(c(t),z) - \grad b_p(c(t),\xi ), \dot c(t) \rangle dt \right |
< \frac{\epsilon}{r}  \int_0^1 \|\dot c(t) \| dt \le \epsilon
\end{multline*}
which yields the claim made in the corollary.
\end{proof}

\begin{corollary}\label{lem:b-xi}
Given $p,p',q\in \wM$ and $\xi\in\ideal$, we have $b_{p'}(q,\xi) = b_p(q,\xi) - b_p(p',\xi)$.
In particular,  all Busemann functions based at $\xi$ coincide up to an additive constant and 
$b_{p'}(q,\xi) =  - b_q(p',\xi)$.
\end{corollary}
\begin{proof}
Given  $p,p',q, z\in \wM$ we obtain 
 $b_{p'}(q,z) = b_p(q,z) - b_p(p',z)$ and taking the limit $z \to \xi $ corollary \ref{cor:c1} yields the claim.
\end{proof}

The following result justifies the terminology `stable manifold' for $W^s(v)$.

\begin{lemma}\label{lem:uniqasym}
Let $(M,g)$ be a closed Riemannian manifold without conjugate points and of hyperbolic type.
Then for each $v \in T^1\wM$, we have
\begin{equation}\label{eqn:Ws-asymptotic}
W^s(v) =\{w \in T^1\wM : c_w \text{ is asymptotic to }  c_v  \}.
\end{equation}
\end{lemma}
\begin{proof}
By Lemma \ref{lem:p-to-xi}, both the left- and right-hand sides of \eqref{eqn:Ws-asymptotic} contain exactly one point from each $T_p^1\wM$.  By Lemma \ref{lem:ct}, each $T_p^1\wM$ contains a $w$ that lies in both the left- and right-hand sides.  The result follows.
\end{proof}

We say that a geodesic $c\colon \RR\to\wM$ connects two points at infinity $\eta,\xi\in\ideal$ if $c(-\infty) := c^-(\infty) = \eta$ and $c(\infty)=\xi$, where $c^-(t) = c(-t)$.

\begin{lemma}[\cite{wK71}]\label{lem:eta-xi}
For every $\eta,\xi\in \ideal$ with $\eta\neq\xi$, there exists a geodesic $c$ connecting $\eta$ and $\xi$.
\end{lemma}
%\gk{this lemma is already in Morse respectively Klingenberg KLI71 for higher dimension. Alternatively on could quote Eberlein who proved it for visibility manifolds}
%\begin{proof}
%By the existence of the background metric of negative curvature, for every \(\eta,\xi\in\ideal\), there is a unique geodesic $\alpha$ with respect to the metric of negative curvature joining \(\eta\) and \(\xi\). 
%\changed{The Morse Lemma gives for each $t>0$ a $g$-geodesic $c_t$ that passes through $B(\alpha(0),R_0)$ and connects $\alpha(-t)$ and $\alpha(t)$.  Reparametrizing so that $c_t(0) \in B(\alpha(0),R_0)$, we can choose $t_k\to\infty$ such that $\dot c_{t_k}(0) \to v\in \pi^{-1}\overline{B(\alpha(0),R_0)}$. Then $c_v$ remains within $R_0$ of the image of $\alpha$ for all time, so $c_v$ is a $g$-geodesic connecting $\eta$ and $\xi$.}
%\end{proof}

The geodesic $c$ in Lemma \ref{lem:eta-xi} is \emph{not} always unique; there may be multiple geodesics connecting $\eta$ and $\xi$, in which case $\eta$ and $\xi$ are in some sense `conjugate points at infinity'; we allow such points in $\ideal$ even though we forbid conjugate points in $\wM$.
In the more restrictive setting of no focal points (in particular, if $M$ has nonpositive curvature), any two distinct geodesics connecting $\eta\neq\xi\in \ideal$ must bound a flat strip in $\wM$, but this is no longer the case in our setting.

Given $v,w\in T^1\wM$, observe that $c_v(\pm\infty) = c_w(\pm\infty)$ if and only if $w\in W^s(v) \cap W^u(v)$ by Lemma \ref{lem:uniqasym}, and so $c_v$ is the unique geodesic joining $c_v(\pm\infty)$ if and only if
\begin{equation}\label{eqn:Wsu-orbit}
W^s(v) \cap W^u(v) = \{\dot{c}_v(t) : t\in \RR \} = \{ f_t v : t\in \RR \}.
\end{equation}
We can also give a characterization in terms of the horospheres:  \(c_v\) is the unique geodesic joining \(c_v(\pm\infty)\) if and only if 
%for every \(x\in H^s(v)\cap H^u(v)\) we have \(x=\pi(v)\).
$H^s(v) \cap H^u(v)$ consists of the single point $\pi(v)$.

In the following we call
\begin{equation}\label{eqn:E}
\begin{aligned}
\mathcal{E} &:= 
\{ v\in T^1\wM : W^s(v) \cap W^u(v) = \{ \dot{c}_v(t) : t\in\RR \} \} \\
&=\{ v \in T^1\wM : H^s(v)\cap H^u(v)\text{ is a single point}\}.
\end{aligned}
\end{equation}
the \emph{expansive set}, which we will use in the definition of $\HHH$ in \S\ref{sec:high}:
From the discussion above, we have $v\in \mathcal{E}$ if and only if $c_v$ is uniquely determined up to parametrization by $c_v(\pm\infty)$.

%\subsection{General results about equilibrium states}\label{sec:es}

\subsection{A uniqueness result using specification}\label{sec:gen}\label{sec:es}

We use an approach to uniqueness of the measure of maximal entropy that goes back to Bowen \cite{rB75}; see Franco \cite{eF77} for the flow case.  This approach relies on the properties of \emph{expansiveness} and \emph{specification}.  We will use versions of these properties that differ slightly from those used by Bowen and Franco, but which keep the essential features; we describe these now, together with a general uniqueness result proved by the first author and D.J.\ Thompson \cite{CT16} that extends Bowen's result to a more nonuniform setting.

Let $(X,F)$ be a continuous flow on a compact metric space.  

\begin{definition}\label{def:NE}
Given $\eps>0$, say that a point $x\in X$ is \emph{expansive at scale $\eps$} if there is $s>0$ such that
\[
\{y\in X : d(f_t x, f_ty) < \eps \text{ for all } t\in \RR \} \subset \{f_t x : t\in [-s,s]\}.
\]
Let $\NE(\eps)$ be the set of points in $X$ that are \emph{not} expansive at scale $\eps$.  The \emph{entropy of obstructions to expansivity at scale $\eps$} is
\begin{equation}\label{eqn:hexp}
\hexp(\eps) = \sup \{ h_\mu(f_1) : \mu\in \MMM_F(X),\ \mu(\NE(\eps))=1\}.
\end{equation}
\end{definition}

\begin{definition}\label{def:specification}
Given $\delta>0$, we say that the flow has
\emph{specification at scale $\delta>0$} if  there exists $\tau=\tau(\delta)$ such that for every $(x_1, t_1)$, $\dots, (x_N, t_N)\in X\times [0,\infty)$  there exist a point $y\in X$ and times $T_1,\dots, T_N\in\RR$ such that $T_{j+1} - (T_j + t_j) \in [0,\tau]$ for every $1\leq j < N$ and
%$\tau_1, \dots, \tau_{N-1}\in[0, \tau]$ such that for $s_0=\tau_0=0$ and $s_j=\sum_{i=1}^j t_i + \sum_{i=1}^{j-1} \tau_i$, we have
\[
f_{T_j}(y)\in B_{t_j}(x_j, \delta):=\Big\{z\in X: \sup_{s\in[0,t_j]}d(f_sx_j, f_sz)<\delta \Big\}
\]
for every $j\in \{ 1,\dots, N\}$.  It is convenient to also use the notation $s_j = T_j + t_j$ for the time at which the orbit of $y$ stops shadowing the orbit of $x_j$, and $\tau_j = T_{j+1} - s_j$ for the time it takes to transition from the orbit of $x_j$ to the orbit of $x_{j+1}$.
%$T_j= s_{j-1}+\tau_{j-1}$ for the time at which the orbit of $y$ is near $x_j$; 
See Figure \ref{fig:specification} for the relationship between the various times.
\end{definition}

\begin{figure}[htbp]
\includegraphics[width=\textwidth]{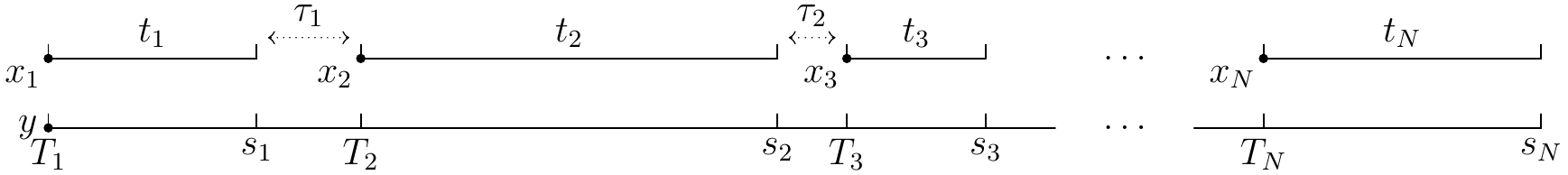}
\caption{Book-keeping in the specification property.}
\label{fig:specification}
\end{figure}

\begin{remark}
It is well-known that the specification property holds for transitive Anosov flows, including the geodesic flow on a smooth negatively curved closed Riemannian manifold; see \cite{rB72} and also Appendix \ref{sec:specification}.
\end{remark}

The following result is proved in \cite[Theorem 2.9]{CT16}.

\begin{theorem}\label{thm:general}
Let $(X,F)$ be a continuous flow on a compact metric space, and suppose that there are $\delta,\eps>0$ with $\eps>40\delta$ such that $\hexp(\eps) < \htop(f_1)$ and the flow has specification at scale $\delta$.  Then $(X,F)$ has a unique measure of maximal entropy.
\end{theorem}

\begin{remark}
The result proved in \cite{CT16} is more general (and more complicated) in two ways: it applies to nonzero potential functions, and it only requires specification to hold for a (sufficiently large) collection of orbit segments.  The version stated here follows from \cite[Theorem 2.9]{CT16} by putting $\phi=0$, $\GGG = \mathcal{D} = X\times \RR^+$, and $\mathcal{P} = \mathcal{S} = \emptyset$.
\end{remark}

\subsection{Limit distribution along periodic orbits}\label{sec:per-limit}

Let $X$ be a compact metric space and $f_t\colon X \to X$ a continuous flow. 
%For \(\delta, T>0\), let \(\mathcal P(T,\delta)\) be a maximal \(\delta\)-separated set of periodic orbits of period less than \(T\) and \(P(T,\delta) := \card \mathcal P(T,\delta) \). 
Given $T>0$, let $\mathcal{P}(T)$ be a maximal set of pairwise non-homotopic periodic orbits with period at most $T$, and let $P(T) = \card \mathcal{P}(T)$.
Let \(\mu_T\) be the invariant probability measure defined by
\begin{equation}\label{eqn:muT}
\int_X \varphi \, d\mu_T = \frac1 {P(T)} \sum\limits_{\gamma \in \mathcal P (T)}
\frac{1}{\ell(\gamma)} \int\limits_{0}^{\ell(\gamma)}  \varphi(\gamma(s)) ds
\end{equation}
for each 
 $\varphi \in C^0(X)$.
%The prove is similar to the proof of \cite[Proposition 6.4]{gK98} 

\begin{definition}\label{def:periodics}
We say that a flow-invariant probability measure $\mu$ is the \emph{limiting distribution of (homotopy classes of) periodic orbits} if $\mu_T\to \mu$ in the weak* topology.  If this occurs in the case when $X = T^1M$ and $f_t$ is the geodesic flow, we also say that $\mu$ is the \emph{limiting distribution of (homotopy classes of) closed geodesics}.
\end{definition}

\begin{remark}\label{rmk:CKW}
In our proofs of Theorems \ref{thm:surfaces} and \ref{thm:higher-dim}, we prove equidistribution of periodic orbits in the sense of Definition \ref{def:periodics}. One could also pursue a stronger equidistribution result by fixing $\delta>0$ and writing $\mathcal{P}(T,\delta)$ for the set of orbits in $\mathcal{P}(T)$ with period in $(T-\delta,T]$; in negative curvature and in nonpositive curvature, the measures corresponding to $\mathcal{P}(T,\delta)$ as in \eqref{eqn:muT} converge to the MME as $T\to\infty$ for every $\delta>0$ \cite{gM04,BCFT}. This stronger equidistribution result, which is crucial for the Margulis asymptotic estimates \cite{gM69,gM04},
turns out to be true in our setting as well, as we prove in \cite{CKW}. However, it requires machinery that we do not develop here; the crucial step is to show that for every $\delta>0$, $\card \mathcal{P}(T,\delta)$ grows with exponential rate $\htop(f_1)$. As we discuss in \S\ref{sec:equidist}, this growth is already known for $\card \mathcal{P}(T)$, which gives an estimate for $\card \mathcal{P}(T,\delta)$ if we restrict to ``most'' $T$, but the stronger estimate is beyond the scope of this paper.
\end{remark}

It is worth observing here that the notion of length does not depend on which representative of a free homotopy class we choose.

\begin{lemma}\label{lem:same-length}
Let $M$ be a manifold with no conjugate points.
If $c_1,c_2$ are closed geodesics in the same free homotopy class, then they have the same length.
\end{lemma}
\begin{proof}
Let $\ell_i$ be the length of $c_i$. Lifting the homotopy to the universal cover $\wM$ gives geodesics $\tilde c_1, \tilde c_2 \colon \RR\to \wM$ such that $d(\tilde c_1(n\ell_1), \tilde c_2(n\ell_2)) = d(\tilde c_1(0), \tilde c_2(0))=: r$ for all integers $n$. Thus
\begin{align*}
n\ell_1 &= d(\tilde c_1(0), \tilde c_1(n\ell_1)) \\
&\leq d(\tilde c_1(0), \tilde c_2(0)) + d(\tilde c_2(0), \tilde c_2(n\ell_2)) + d(\tilde c_2(n\ell_2), \tilde c_1(n\ell_1)) = 2r + n\ell_2,
\end{align*}
where the equalities use the assumption of no conjugate points so that geodesics minimize distances in $\wM$. Dividing by $n$ and sending $n\to\infty$ gives $\ell_1 \leq \ell_2$, and by symmetry this suffices.
\end{proof}

\subsection{Residually finite fundamental groups}\label{sec:algebra}

\begin{definition}
A group $G$ is \emph{residually finite} if the intersection of its finite index subgroups is trivial.
\end{definition}

For surfaces we have the following result which was first proved by Baumslag \cite{gB62} and then Hempel \cite{jH72} gave an alternative proof.

\begin{theorem}\label{prop:hyp-metric}
Every surface has residually finite fundamental group.
\end{theorem}
Later on, Hempel \cite{jH87} proved that fundamental groups of three manifolds are residually finite. It is an open problem whether every manifold supporting a negatively curved metric has a residually finite fundamental group \cite{gA14}.

For our purposes we need the following implication of a manifold having residually finite fundamental group.
\begin{proposition}\label{prop:large-N}
Let $M$ be a smooth Riemannian manifold and suppose that  $\pi_1(M)$ is residually finite.  Then for every $R>0$  there is a smooth Riemannian manifold $N$ and a locally isometric covering map $p\colon N\to M$ such that the injectivity radius of $N$ is at least $R$.
\end{proposition}
\begin{proof}
Fix a point $x$ in the universal cover $\wM$, and consider the finite set $Z = \{ \gamma \in \pi_1(M) \setminus \{e\} : \gamma(x) \in B(x,2R)\}$.  Since $\pi_1(M)$ is residually finite, for each $\gamma\in Z$ there is a finite index subgroup $G_\gamma < \pi_1(M)$ such that $\gamma \notin G_\gamma$.  Then $G = \bigcap_{\gamma\in Z} G_\gamma < \pi_1(M)$ is a finite index subgroup such that $d(x,\gamma(x)) \geq 2R$ for all nontrivial $\gamma\in G$.
In particular, $N = \wM / G$ defines a compact manifold that is a finite cover of \(M\) has injectivity radius at least $R$. To see this, we can consider without loss of generality that \(x\in D\) where \(D\) is a fundamental domain corresponding to the covering \(\wM\to M\). Then there is a fundamental domain of \(\wM\to N\) that contains all \(\gamma D\) where \(\gamma\in Z\) and by the definition of \(Z\), this  domain contains a ball of radius  \(\geq R\).
\end{proof}

\section{A class of manifolds with unique measure of maximal entropy}

\label{sec:high}

%\subsection{Statement of the result}

%A large part of Theorem \ref{thm:surfaces} is a consequence of the following general result, which we prove in \S\ref{sec:proof} using Theorem \ref{thm:general}.  In \S\ref{sec:surfaces-pf} we explain how Theorem \ref{thm:surfaces} is deduced from this result.
%Before stating the result, 

Now we can define the class $\mathcal{H}$ of manifolds to which Theorem \ref{thm:higher-dim} applies.  

%we recall that if $\mu$ is a probability measure on $T^1M$, then there is a unique $\sigma$-finite measure $\tmu$ on $T^1\wM$ that is $\pi_1(M)$-invariant and whose restriction to every fundamental domain projects to $\mu$ under the universal covering map.  We refer to $\tmu$ as the \emph{lift} of $\mu$, and note that it is flow-invariant if and only if $\mu$ is.

\begin{definition}\label{def:H}
Let $\mathcal{H}$ denote the class of closed smooth Riemannian manifolds $(M,g)$ without conjugate points such that the following conditions are satisfied.
%\begin{theorem}\label{thm:higher-dim}
%Let $(M,g)$ be a smooth closed Riemannian manifold without conjugate points.  Suppose also that
\begin{enumerate}[label=\textbf{\upshape{(H\arabic{*})}}]
\item\label{H1} $M$ supports a Riemannian metric $g_0$ for which all sectional curvatures are negative; 
\item\label{H2} $(M,g)$ has the divergence property;
\item\label{H3} the fundamental group $\pi_1(M)$ is residually finite; 
\item\label{H4} $\sup\{h_\mu(f_1) : \mu\in \MMM_F(T^1M), \tmu(\mathcal{E})=0\} < \htop(F)$, where $\mathcal{E}$ is the expansive set 
defined in \eqref{eqn:E} and $\tmu$ is the lift of $\mu$ given in \eqref{eqn:tmu-lift}; 
\end{enumerate}
%Then the geodesic flow for $g$ on $T^1M$ has a unique measure of maximal entropy $\mu$.  
%The measure $\mu$ is ergodic and is the limiting distribution of closed geodesics.
%\end{theorem}
\end{definition}

%As we said above, for two and three dimensional manifolds, the topological conditions of Theorem \ref{thm:higher-dim} are known to be satisfied \cite{gB62, jH72, jH87}, and thus the conclusions hold for every metric without conjugate points as soon as the entropy gap condition and the existence of background metric of negative curvature can be verified.

%\begin{example}\label{eg:hyperbolic}
%If $M$ admits a hyperbolic metric -- that is, if $M$ is diffeomorphic to $\mathbb{H}^n / \Gamma$ for some cocompact discrete isometry group $\Gamma$\vc{Is this the right definition of `hyperbolic metric'?} -- then the corresponding geodesic flow is Anosov, so the first condition is satisfied, and $\pi_1(M)$ is known to be residually finite, so Theorem \ref{thm:higher-dim} applies to every metric without conjugate points.  In particular, every two-dimensional manifold with genus $\geq 2$ admits a hyperbolic metric, and thus Theorem \ref{thm:surfaces} is a corollary of Theorem \ref{thm:higher-dim} together with the fact from \cite{LLS16} that an ergodic measure of maximal entropy for geodesic flow over a surface has the Bernoulli property.\vc{Also need the entropy gap, and an argument for fully supported}
%\end{example}

We show below that every closed surface without conjugate points and genus  $\geq 2$ satisfies \ref{H1}--\ref{H4}; this is the key to deducing Theorem \ref{thm:surfaces} from Theorem \ref{thm:higher-dim}.

\begin{remark}

As observed in Remark \ref{rmk:g0}, there are rank 1 manifolds of nonpositive curvature for which \ref{H1} fails, but which have unique measures of maximal entropy by \cite{gK98}.  Thus this condition places a genuine topological restriction on the class of manifolds contained in $\HHH$.

In higher dimensions, the status of the other conditions is less clear; we are not currently aware of any examples for which \ref{H1} holds but any of \ref{H2}--\ref{H4} fails.  It is known that \ref{H3} holds whenever $\dim M \leq 3$
\cite{gB62,jH72,jH87}, and it is an open problem in geometric group theory to determine whether \ref{H1} implies \ref{H3} in general \cite{gA14}.

Following similar arguments to those in \cite[\S8]{BCFT}, condition \ref{H4} can be verified under the following assumptions (we omit the proof):
\begin{enumerate}
\item Conditions \ref{H1}--\ref{H3} hold;
\item the expansive set $\mathcal{E}$ has non-empty interior;
\item the finite cover $N$ of $M$ constructed in the next section has the property that its geodesic flow is entropy-expansive at scale $10\delta$, where $\delta$ is given in in \eqref{eqn:R3} below.  
\end{enumerate}
\end{remark}

%\begin{remark}\label{eg:3-dim}
%%\cite{jH87}
%%If $\dim M=3$, then $\pi_1(M)$ is residually finite, and so Theorem \ref{thm:higher-dim} applies to every 3-manifold without conjugate points that admits a negatively curved metric and for which the entropy gap can be verified.
%Condition \ref{H3} holds whenever $\dim M \leq 3$ \cite{gB62, jH72, jH87}.
%There are examples of 3-manifolds that admit a metric $g$ with no conjugate points and that satisfy the entropy gap condition \ref{H4}, but do not admit a metric of negative curvature.
%\end{remark}

%\subsection{Proof of Theorem \ref{thm:surfaces}}\label{sec:surfaces-pf}

\begin{proof}[Proof of Theorem \ref{thm:surfaces} assuming Theorem \ref{thm:higher-dim}]

Let $M$ be a closed surface of genus $\geq 2$, and $g$ a metric on $M$ with no conjugate points.  We claim that conditions \ref{H1}--\ref{H4} are satisfied.  Indeed, \ref{H1} is a standard result; \ref{H3} is Theorem \ref{prop:hyp-metric}; and \ref{H2} was proved in \cite{lwG56}.  The proof of \ref{H4}
is a consequence of the following proposition.

\begin{proposition}\label{prop:entropy-gives-expansive}
Let $M$ be a surface of genus $\geq 2$ without conjugate points, and  $\mu\in \MMM_F(T^1M)$ an ergodic measure with $h_\mu(f_1) > 0$. Then $\tmu(T^1\wM \setminus \mathcal{E}) = 0$.
\end{proposition}
\begin{proof}
 Let $\mu\in \MMM_F(T^1M)$ be a ergodic measure with $h_\mu(f_1) > 0$. 
This implies  by Ruelle's inequality that   $\mu$-a.e.
$v \in T^1M$  has nonzero Lyapunov exponents, and hence by Pesin theory \(v\) has transverse stable and unstable leaves. 
Using Lemma \ref{lem:uniqasym}, the stable and unstable manifolds of $\mu$-a.e.\ \(v\) correspond to normal fields of 
the stable and unstable horospheres, and thus these horospheres intersect in a single point, so $\mu$-a.e.\ $v\in\mathcal{E}$.
\end{proof}

Finally, the topological entropy of the geodesic flow is positive by Lemma \ref{lem:h>0}, so Proposition \ref{prop:entropy-gives-expansive} establishes \ref{H4}.
%the first condition of Theorem \ref{thm:higher-dim}.  %Moreover, every such $M$ admits a hyperbolic metric, so the second condition is satisfied, and the third is as well by Theorem \ref{prop:hyp-metric}. 
Thus Theorem \ref{thm:higher-dim} applies to every surface of higher genus with no conjugate points, giving a unique measure of maximal entropy $\mu$, which is ergodic and is the limiting distribution of closed geodesics. It follows from \cite{LLS16} that $\mu$ is Bernoulli which concludes the proof of Theorem \ref{thm:surfaces}.
\end{proof}

\section{Uniqueness and equidistribution}\label{sec:proof}

\subsection{Proof of uniqueness}%Theorem \ref{thm:higher-dim}}

In this section we prove the first part of Theorem \ref{thm:higher-dim} by using Theorem \ref{thm:general} to establish uniqueness of the MME when $(M,g)$ is a smooth closed Riemannian manifold without conjugate points satisfying \ref{H1}--\ref{H4}.

The first step is to pass to an appropriate finite cover.
Let $A$ be given by \eqref{lem:unif-equiv}, and $R_0$ by the Morse lemma.  Fix $R_1 > 3AR_0$ and let $R_2$ be given by Lemma \ref{lem:endpts-suffice}; observe that the proof in Appendix \ref{sec:Morse} gives $R_2 = 4R_0 + (6A^2+1)R_1$.  Let
\begin{equation}
%\begin{multline*}
\label{eqn:R3}
%\delta = \max \{d(\dot c_1(0), \dot c_2(0)) : c_1,c_2\colon \RR\to M \text{ are geodesics} \\ \text{such that } d(c_1(0),c_2(0)) \leq R_2\}
\delta = R_2 + 2 = 4R_0 + (6A^2+1)R_1 + 2.
\end{equation}
%\end{multline*}
and fix $\eps > 40\delta$.
By \ref{H3} and Proposition \ref{prop:large-N}, there is a finite cover $N$ of $M$ whose injectivity radius exceeds $3\eps$.
Observe that the covering map $p\colon N\to M$ naturally extends to a finite-to-1 semi-conjugacy from the geodesic flow on $T^1N$ to the geodesic flow on $T^1M$ which implies that

\begin{lemma}\label{lem:p*}
The pushforward map $p_* \colon \MMM_F(T^1 N) \to \MMM_F(T^1 M)$ is surjective and entropy-preserving.
\end{lemma}

It follows from Lemma \ref{lem:p*} that if the geodesic flow on $T^1N$ has a unique measure of maximal entropy, then so does the geodesic flow on $T^1 M$. Ergodicity follows from uniqueness because otherwise every ergodic component will be an MME. 
Similarly, if the unique MME on $T^1N$ is the limit distribution of periodic orbits, then the flow on $T^1N$ satisfies $\lim_{T\to\infty} \frac 1T \log P(T,\delta) = \htop(f_1)$, and since the semi-conjugacy is finite-to-1, the same is true of the geodesic flow on $T^1M$, so Proposition \ref{prop:limidis} gives the corresponding result for $T^1M$.
%To see that the unique MME on \(T^1N\) being the limit distribution of closed orbit implies the same result on \(T^1M\), we first observe since \(N\) is a finite cover then the rate of growth of periodic orbits on \(T^1M\) and \(T^1N\) are the same and is the topological entropy. Then from Proposition \ref{prop:limidis}, the unique measure of maximal entropy on \(T^1N\) being limit distribution of closed orbit implies the same result of the measure of maximal entropy on \(T^1M\). 

Thus prove the claims in Theorem \ref{thm:higher-dim} regarding uniqueness and closed geodesics, it suffices to prove them for geodesic flow on $T^1N$, which we will do using Theorem \ref{thm:general}.  From now on we consider $X = T^1N$ and let $F$ be the geodesic flow.

\begin{lemma}\label{lem:NE}
If $v\in \NE(\eps) \subset T^1N$, then any lift $\widetilde v$ of $v$ to $T^1\widetilde N = T^1\wM$ has the property that $\widetilde v\notin \mathcal{E}$.
In particular, if $\mu\in \MMM_F(T^1N)$ is such that $\mu(\NE(\eps))=1$, 
then $\tmu(\mathcal{E})=0$.
\end{lemma}
\begin{proof}
If $v\in \NE(\eps)$, then for every $s>0$, there is $w\in T^1 N$ such that $w\notin \{f_t v : t\in [-s,s]\}$, but $d(f_t v, f_t w) < \eps$ for all $t\in \RR$.  
Given a lift $\widetilde v \in T^1 \widetilde N = T^1 \wM$ of $v$, let $\widetilde w\in T^1\widetilde N$ be a lift of $w$ with $d(\widetilde v, \widetilde w) < \eps$. Then for all $t\in \RR$ there is a unique $\gamma(t) \in \pi_1(N)$ such that $d(f_t \widetilde v, \gamma(t) f_t \widetilde w) < \eps$; existence follows since $d(f_t v, f_t w) < \eps$, and uniqueness follows since $2\eps$ is smaller than the injectivity radius of $N$. The function $\gamma\colon \RR\to \pi_1(N)$ is continuous and thus constant on $\RR$, so $d(f_t \widetilde v, f_t \widetilde w) < \eps$ for all $t\in\RR$.
In particular, taking $s>\eps$ we conclude that $\widetilde v,\widetilde w$ are tangent to distinct geodesics between the same points on $\ideal$, and thus $\widetilde v\notin \mathcal{E}$.
%Let $\mu\in \MMM_F(T^1N)$ be such that $\mu(\NE(\eps))=1$.   Then for $\mu$-a.e.\ $v\in T^1 N$ and 
The claim regarding $\mu$ and $\tmu$ follows immediately.
\end{proof}

It follows from Lemma \ref{lem:NE} and 
%the first hypothesis in Theorem \ref{thm:higher-dim} that
Condition \ref{H4} that
\[
\hexp(\eps) \leq \sup\{h_\mu(f_1) : \mu\in \MMM_F(T^1M), \tmu(\mathcal{E})=0\} < \htop(F),
\]
which verifies the entropy gap condition that is needed for Theorem \ref{thm:general}.

\begin{proposition}\label{prop:specification}
The geodesic flow of $(N,g)$ has specification at scale $\delta$.
\end{proposition}

This is a consequence of Theorem \ref{specgeoflow} and Remark \ref{rmk:spec} below.
The basic idea is to use the Morse Lemma to go from orbit segments for $F$ to orbit segments for $F^0$, then use the specification property for $F^0$ to find a single $g_0$-geodesic that shadows each of these in turn, and finally to use the Morse Lemma again to show that the $g$-geodesic with the same endpoints in $\wM$ shadows the original sequence of orbit segments.

Once specification at scale $\delta$ has been proved, Theorem \ref{thm:general} implies that the geodesic flow on $T^1N$ has a unique measure of maximal entropy $\mu$.
%, and by Lemma \ref{lem:p*}, so does the geodesic flow on $T^1 M$.
%Ergodicity of $\mu$ follows immediately from uniqueness, since if $\mu$ was not ergodic, then every component of its ergodic decomposition would be a measure of maximal entropy.  

\subsection{Specification}\label{sec:specification}

Let $M$ be a smooth closed Riemannian manifold without conjugate points satisfying \ref{H1} and \ref{H2}.
Let $R_1,R_2,\delta$ be as given above. This section is devoted to the proof of the following.

\begin{theorem} \label{specgeoflow}
There exist $\tau,\tau'>0$ such that given $(v_1,t_1),\dots, (v_k,t_k) \in T^1M\times (0,\infty)$ and $T_1,\dots, T_k \in \RR$ with $T_{j+1} - T_j \geq t_j + \tau$ for all $1\leq j < k$, there are $\hat{T}_j \in [T_j-\tau',T_j]$ and  $w\in T^1M$ such that for all $1\leq j\leq k$, we have $f_{\hat{T}_j}w \in B_{t_j}(v_j,\delta)$. 
\end{theorem}

\begin{remark}\label{rmk:spec}
The conclusion of Theorem \ref{specgeoflow} is a mild strengthening of the specification property from Definition \ref{def:specification}.  To deduce that property from this one, observe that the property here remains true if $\tau,\tau'$ are replaced by $\max(\tau,\tau')$, and then choosing $T_{j+1} = T_j + t_j + \tau$, the times $\hat T_j$ satisfy $\hat T_{j+1} - (\hat T_j + t_j) \leq T_{j+1} - (T_j - \tau + t_j) \leq 2\tau$.
\end{remark}

Note that it suffices to prove Theorem \ref{specgeoflow} in the case when $T_{j+1} = T_j + t_j + \tau$; to reduce the general case to this one, replace $t_j$ by $T_{j+1} - T_j - \tau \geq t_j$ and observe that this does not weaken the condition on $w$.

Let $F^0 = \{f_t^0\}$ denote the geodesic flow in the background (negatively curved) Riemannian metric.  We identify $T^1M$ for $g_0$ with $T^1M$ for $g$ in the natural way.  To establish specification for $F$, we use the hyperbolicity properties of $F^0$ as in \cite[\S4]{BCFT} together with the Morse Lemma.  

Let $W^{s}$ and $W^{u}$ denote the strong stable and unstable foliations of $T^1M$ for the background flow $F^0$.  (We follow the notation in \cite{BCFT} rather than that used in \S\ref{sec:background}, where this notation referred to the weak foliations; in this section we will not use any of the foliations for the flow $F$.)
Equip the leaves of these foliations with the intrinsic metrics $d^s$ and $d^u$ defined by pulling back the metrics that the horospheres inherit from the Riemannian metric.  Let $W^{cs}$ denote the foliation whose leaves have the form $\{f_t^0(W^s(v)) : t\in \RR\}$, and define $d^{cs}$ on each leaf of $W^{cs}$ locally by $d^{cs}(v,w) = |t| + d^s(f_t^0 v, w)$, where $t$ is such that $f_t^0 v\in W^s(w)$.  Given $\rho>0$, let $W_\rho^u(v)$ denote the $d^u$-ball in $W^u(v)$ of radius $\rho$ centered at $v$, and similarly for $W_\rho^{cs}$.

The foliations $W^{cs}$ and $W^u$ have the following local product structure property:
there are $\kappa\geq 1$ and $\delta>0$ such that for every $\eps\in (0,\delta]$ and all $w_1,w_2 \in B(v,\eps)$, the intersection $W_{\kappa \eps}^u(w_1) \cap W_{\kappa \eps}^{cs}(w_2)$ contains a single point, which we denote by $[w_1,w_2]$, and this point satisfies
\begin{align*}
d^u(w_1, [w_1,w_2]) &\leq \kappa d_1(w_1,w_2), \\
d^{cs}(w_2, [w_1,w_2]) &\leq \kappa d_1(w_1,w_2).
\end{align*}

\begin{proposition} \label{prop:denseleaf}
For every $\rho>0$, there exists $T>0$ such that for every $v,w\in T^1M$, we have $f_t^0(W_\rho^u(v)) \cap W_\rho^{cs}(w) \neq\emptyset$ for every $t\geq T$.
\end{proposition}
\begin{proof}
Since every leaf of $W^u$ is dense in $T^1M$, a simple compactness argument as in \cite[Lemma 8.1]{CFT18} gives $R>0$ such that $W_{R-\rho}^u(f_t^0 v)$ is $\rho/\kappa$-dense in $T^1M$ for every $v\in T^1M$ and $t\in \RR$.  By the local product structure, we have $W_R^u(f_t^0 v) \cap W_\rho^{cs}(w) \neq \emptyset$.  Since leaves of $W^u$ are uniformly expanded by $f^0_t$,  there exists $T>0$ such that $f_t^0(W_\rho^u(v)) \supset W_R^u(f_t^0 v)$ for every $v\in T^1M$ and $t\geq T$, which completes the proof.
 \end{proof}

Since leaves of $W^u$ are uniformly expanded by $f_t^0$, there is $\lambda\in (0,1)$ such that whenever $v,w$ lie in the same leaf of $W^u$, we have
\begin{equation}\label{eqn:u-contracts}
d^u(f_{-t}^0 v, f_{-t}^0 w) \leq \lambda d^u(v,w)
\quad\text{for all }t\geq 1.
\end{equation}
Let $\rho'>0$ be sufficiently small that $R_1 > 3A(R_0 + \rho')$, and let
\begin{equation}\label{eqn:rho'}
\rho = \rho' (1-\lambda)/2.
\end{equation}
By Proposition \ref{prop:denseleaf}, there exists $T\geq 1$ such that $f_t(W_{\rho}^u(v))$ intersects $W_{\rho}^{cs}(w)$ whenever $t\geq T$.
We will prove Theorem \ref{specgeoflow} with $\tau=AT$
%, where $A$ is the constant relating $d^0$ and $d$.
and $\tau' = 2\tau_0$, where $\tau_0 = 2\tau + 7AR_0 + 4A\rho'$.

\begin{definition}[Correspondence between orbit segments]
The following procedure defines a map $E\colon T^1M\times (0,\infty) \to T^1M\times (0,\infty)$ with the property that if $(v,t)$ represents an $F$-orbit segment, then $E(v,t)$ represents an $F^0$-orbit segment that shadows it to within the scale given by the Morse Lemma.
\begin{enumerate}
\item Given $(v,t)\in T^1M\times (0,\infty)$, the corresponding $F$-orbit segment projects to the $g$-geodesic segment $c_v([0,t])$.  
\item Let $x,y\in \wM$ be the endpoints of a lift of $c_v([0,t])$ to $\wM$.
\item Let $s = d^0(x,y)$, and let $\alpha\colon [0,s]\to M$ be a $g_0$-geodesic that lifts to a segment running from $x$ to $y$.
\item Let $E(v,t) = (\dot\alpha(0),s)$.
\end{enumerate}
\end{definition}

\begin{figure}[htbp]
\includegraphics[width=\textwidth]{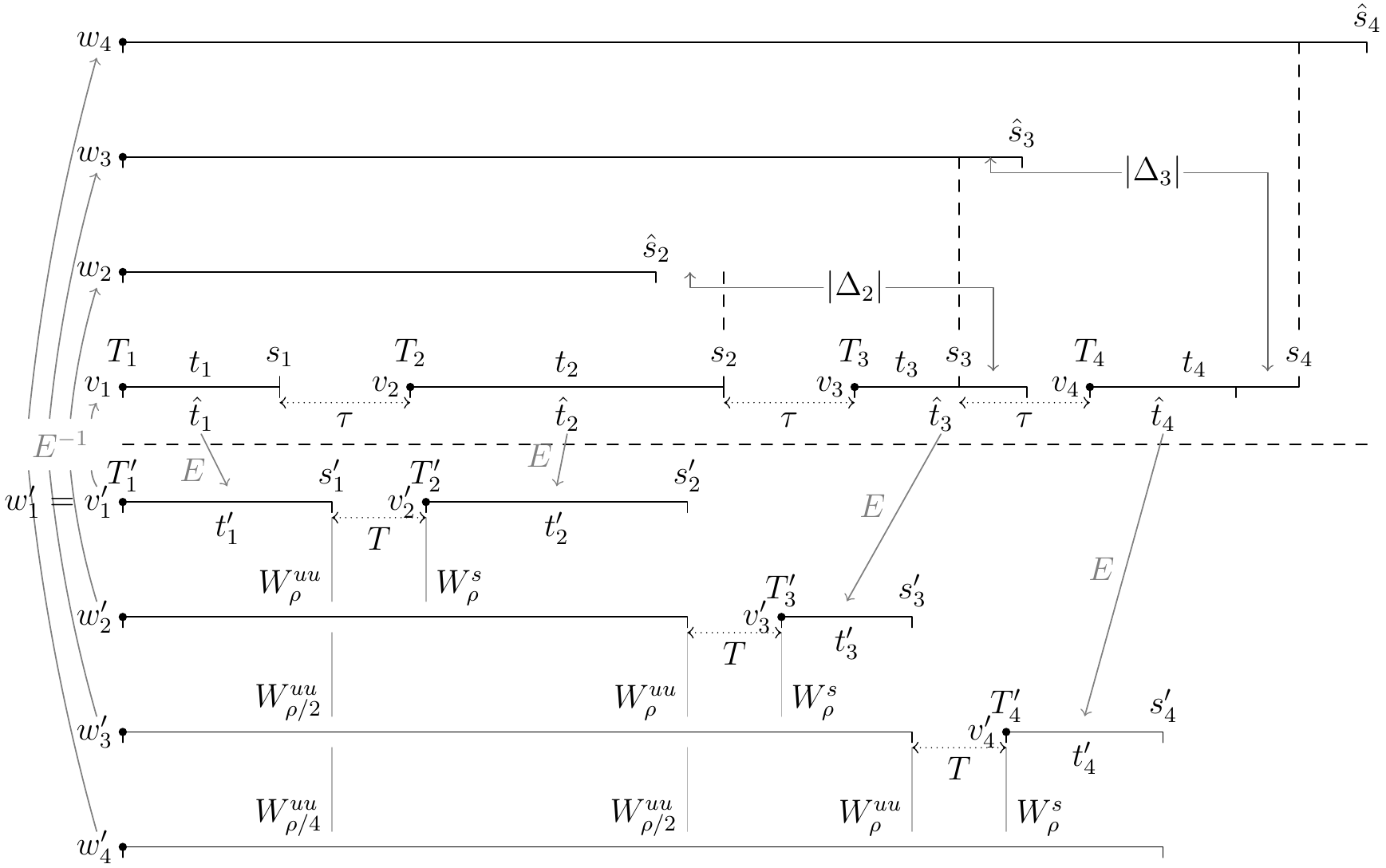}
\caption{Proving specification.  Horizontal lines correspond to orbit segments with length given by the number in the middle of the line; labels above endpoints of lines represent time.  The top half of the picture corresponds to $F,g,d$, the bottom half to $F^0,g_0,d^0$; the two halves are related by specific applications of the map $E$ and its inverse.}
\label{fig:spec-2}
\end{figure}

Fix $(v_1, t_1), \dots, (v_k, t_k)\in T^1M\times (0,\infty)$ and $T_1,\dots, T_k,s_1,\dots,s_k \in \RR$ with $T_1=0$ and
\[
s_j = T_j + t_j, \qquad
T_{j+1} = s_j + \tau.
\]
%Assume without loss of generality that $T_1=0$.
We define sequences $\hat t_i, v_i',t_i',w_i',s_i',T_i',\hat s_i,\Delta_i$ recursively (see Figure \ref{fig:spec-2}).
\begin{itemize}
\item Let $\hat t_1 = t_1$, $\hat s_1 = s_1$, and $\Delta_1 = T_1' = 0$.
\item Let $(w_1',s_1') = (v_1', t_1') = E(v_1,\hat t_1)$.
\end{itemize}
Now fix $j\geq 2$ and suppose all terms have been defined for $1\leq i < j$.
\begin{itemize}
\item Let $\hat t_j = t_j + \Delta_{j-1}$ and $(v_j',t_j') = E(v_j,\hat t_j)$.
\item Let $T_{j}' = s_{j-1}' + T$ and $s_j' = T_j' + t_j'$.
\item Use Proposition \ref{prop:denseleaf} to get $w_j'$ such that
\[
f^0_{T_j'}(w_j') \in f^0_T(W_\rho^u(f^0_{s_{j-1}'} w_{j-1}')) \cap W_\rho^{cs}(v_j').
\]
\item Put $(w_j, \hat{s}_j) = E^{-1}(w_j',s_j')$ and $\Delta_j = s_j - \hat{s}_j$.
\end{itemize}

Observe that $d^{cs}(f^0_{T_j'} w_j', v_j') \leq \rho$ and for all $1\leq i< j\leq k$,
\begin{equation}\label{eqn:du-ij}
d^u(f^0_{s_i'} w'_{j}, f^0_{s'_i} w'_{j-1}) \leq \lambda^{j-1-i} \rho.
\end{equation}
For each $1\leq i < k$, summing \eqref{eqn:du-ij} over $j$ from $i$ to $k-1$ gives
\begin{equation}\label{eqn:duw'}
d^u(f_{s_i'}^0 w_k', f_{s_i'}^0 w_i') \leq \sum_{j=i+1}^{k} \lambda^{j-1-i} \rho < \rho(1-\lambda)^{-1}.
\end{equation}
Let $\alpha_j\colon [0,s_j']\to \wM$ be a $g_0$-geodesic corresponding to $(w_j',s_j')$.
%Let $c\colon [0,\hat{s}_k]\to \wM$ be the $g$-geodesic corresponding to $(w_k,\hat{s}_k)$.  
Then for each $i$, there is a $g_0$-geodesic $\beta_i\colon [0,t_i']\to \wM$ corresponding to $(v_i',t_i')$ such that for every $1\leq i\leq j\leq k$ and every $0\leq t\leq t_i'$, we have
\begin{equation}\label{eqn:beta-alpha}
d^0(\beta_i(t), \alpha_j(T_i' + t)) 
\leq d^{cs}(f_{T_i'}^0 w_i', v_i') + d^u(f_{s_i'}^0 w_k', f_{s_i'}^0 w_i')
< \rho'.
\end{equation}
In particular, writing $x_i = \beta_i(0), y_i = \beta_i(t_i')\in \wM$, we have
\begin{equation}\label{eqn:xiyi}
d^0(x_i, \alpha_j(T_i')) < \rho'
\text{ and }
d^0(y_i,\alpha_j(s_i')) < \rho'
\text{ for every }1\leq i\leq j\leq k.
\end{equation}
Let $b_i\colon [0,\hat{t}_i]\to \wM$ be the $g$-geodesic connecting $x_i$ and $y_i$; note that $b_i$ corresponds to the $F$-orbit segment $(v_i,\hat{t}_i)$.  Let $c_j\colon \RR\to \wM$ be the $g$-geodesic with the property that
\begin{equation}\label{eqn:c-alpha}
c_j(0) = \alpha_j(0)
\text{ and }
c_j(\hat s_j) = \alpha_j(s_j').
\end{equation}
We prove Theorem \ref{specgeoflow} by showing that $\dot{c}_j(0)$ has the desired shadowing properties.

First note that by the Morse Lemma, for each $1\leq i\leq j\leq k$ there are $\tilde T_i^j, \tilde s_i^j\in\RR$ such that
\begin{equation}\label{eqn:T-s-tilde}
d(\alpha_j(T_i'),c_j(\tilde T_i^j)) \leq AR_0
\text{ and }
d(\alpha_j(s_i'), c_j(\tilde s_i^j)) \leq AR_0.
\end{equation}
We can take $\tilde T_0^j=0$ and $\tilde s_j^j = \hat s_j$.  Using \eqref{eqn:xiyi} gives
\begin{equation}\label{eqn:xycj}
d(x_i,c_j(\tilde T_i^j)) \leq A(R_0+\rho')
\text{ and }
d(y_i,c_j(\tilde s_i^j)) \leq A(R_0+\rho').
\end{equation}
This, together with Lemma \ref{lem:endpts-suffice}, will establish the necessary shadowing properties once we have proved the following lemmas relating $\tilde s_i^j,  \hat s_i, s_i$ and $\tilde T_i^j, T_i$.

%understand the relationship between the values of $\tilde s_i^j$,  $\hat s_i$, and $s_i$.  
%We claim that there is a uniform constant (independent of $i,j$) such that for every $1\leq i\leq j\leq k$, all three of these values are within a uniform constant of each other.

\begin{lemma}\label{lem:tilde-hat}
For every $1\leq i\leq j\leq k$, we have $|\tilde s_i^j - \hat s_i| \leq A(R_0 + 2\rho')$.
\end{lemma}

\begin{lemma}\label{lem:tilde-tilde}
For every $1\leq i\leq j\leq k$, we have
\begin{gather}
\label{eqn:s-to-T}
A^{-1} T - 2AR_0 \leq \tilde T_i^j - \tilde s_{i-1}^j \leq \tau + 2AR_0, \\
\label{eqn:T-to-s}
|\tilde s_i^j - (\tilde T_i^j + \hat t_i)| \leq 2A(R_0+\rho').
\end{gather}
\end{lemma}

\begin{lemma}\label{lem:hat-s}
For every $1\leq i\leq k$, we have $|\tilde s_i^k - s_i| \leq \tau + 5AR_0 + 4A\rho'$ and $|\tilde T_i^k - T_i| \leq 2\tau + 7AR_0 + 4A\rho'$.
\end{lemma}

In what follows, we will repeatedly use the following elementary consequence of the triangle inequality.

\begin{lemma}\label{lem:quad}
If $(X,d)$ is a metric space, then for every $w,x,y,z\in X$ we have
\[
|d(w,x) - d(y,z)| \leq d(w,y) + d(x,z).
\]
\end{lemma}
\begin{proof}
The triangle inequality gives $d(w,x) \leq d(w,y) + d(y,z) + d(z,x)$, so $d(w,x) - d(y,z) \leq d(w,y) + d(z,x)$, and the other inequality is similar.
\end{proof}

Before proving Lemmas \ref{lem:tilde-hat}--\ref{lem:hat-s}, we demonstrate how they complete the proof of Theorem \ref{specgeoflow}.  Since $\tilde s_i^k - \tilde T_i^k = d(c_k(\tilde T_i^k), c_k(\tilde s_i^k))$ and $t_i = d(x_i,y_i)$, we can apply Lemma \ref{lem:quad} to these four points to deduce that
\[
|(\tilde s_i^k - \tilde T_i^k) - t_i| \leq d(c_k(\tilde T_i^k),x_i) + d(c_k(\tilde s_i^k),y_i) \leq 2A(\rho_0 + \rho'),
\]
where the last inequality uses \eqref{eqn:xycj}.  Thus
\[
d(c_k(\tilde T_i^k + t_i), y_i)
\leq |\tilde s_i^k - (\tilde T_i^k + t_i)| + d(c_k(\tilde s_i^k),y_i)
\leq 3A(\rho_0 + \rho') < R_1,
\]
so Lemma \ref{lem:endpts-suffice} gives
\begin{equation}\label{eqn:c-shadows}
d(c_k(\tilde T_i^k + t),b_i(t)) < R_2 \text{ for all } 0\leq i\leq k \text{ and } 0\leq t\leq t_i,
\end{equation}
where $R_2 = 4R_0 + (6A^2+1)R_1$.  Taking $\hat T_i = \tilde T_i^k$ gives the desired shadowing property, with $\hat T_i = [T_i - \tau_0, T_i + \tau_0]$ by Lemma \ref{lem:hat-s}, where $\tau_0 = 2\tau + 7AR_0 + 4A\rho'$.  Replacing $T_i$ by $T_i+\tau_0$ gives the form of the property stated in the theorem.

Now we prove Lemmas \ref{lem:tilde-hat}--\ref{lem:hat-s}.

\begin{proof}[Proof of Lemma \ref{lem:tilde-hat}]
Since $\tilde s_i^j = d(c_j(0),c_j(\tilde s_i^j))$ and $\hat s_i=d(c_i(0),c_i(\hat s_i))$, we can apply Lemma \ref{lem:quad} to these four points to get
\begin{align*}
|\tilde s_i^j - \hat s_i| &\leq d(c_j(0),c_i(0)) + d(c_i(\hat s_i),c_j(\tilde s_i^j)) \\
&= d(\alpha_j(0),\alpha_i(0)) + d(\alpha_i(s_i'),c_j(\tilde s_i^j)) \\
&\leq A\rho' + d(\alpha_i(s_i'),\alpha_j(s_i')) + d(\alpha_j(s_i'),c_j(\tilde s_i^j))
\leq 2A\rho' + AR_0,
\end{align*}
where the last inequality uses \eqref{eqn:duw'} to compare $\alpha_i$ and $\alpha_j$, and \eqref{eqn:T-s-tilde} to compare $\alpha_j$ and $c_j$.
\end{proof}

\begin{proof}[Proof of Lemma \ref{lem:tilde-tilde}]
Since $\tilde T_i^j - \tilde s_{i-1}^j = d(c_j(\tilde s_{i-1}^j),c_j(\tilde T_i^j))$, we apply Lemma \ref{lem:quad} to these two points and $\alpha_j(s_{i-1}'),\alpha_j(T_i')$ to get
\begin{multline*}
|\tilde T_i^j - \tilde s_{i-1}^j - d(\alpha_j(s_{i-1}'),\alpha_j(T_i'))| \\
\leq d(c_j(\tilde s_{i-1}^j), \alpha_j(s_{i-1}')) + d(c_j(\tilde T_i^j),\alpha_j(T_i')) \leq 2AR_0.
\end{multline*}
Since $d^0(\alpha_j(s_{i-1}'),\alpha_j(T_i')) = T_i' - s_{i-1}' = T$, \eqref{lem:unif-equiv} gives us the bounds $A^{-1}T \leq d(\alpha_j(s_{i-1}'),\alpha_j(T_i')) \leq AT=\tau$, which proves \eqref{eqn:s-to-T}.
To prove \eqref{eqn:T-to-s}, we observe that $\tilde s_i^j - \tilde T_i^j = d(c_j(\tilde T_i^j), c_j(\tilde s_i^j))$ and $d(x_i,y_i) = \hat t_i$, so we apply Lemma \ref{lem:quad} to these four points, obtaining
\[
|(\tilde s_i^j - \tilde T_i^j) - \hat t_i| \leq d(c_j(\tilde T_i^j), x_i) + d(c_j(\tilde s_i^j),y_i),
\]
and then \eqref{eqn:xycj} completes the proof.
\end{proof}

\begin{proof}[Proof of Lemma \ref{lem:hat-s}]
Recall that $\Delta_i = s_i - \hat s_i$ and $\hat t_i = t_i + \Delta_{i-1}$, so $\hat t_i - t_i = \Delta_{i-1} = s_{i-1} - \hat s_{i-1}$.  From Lemma \ref{lem:tilde-tilde} we have
\[
|\tilde s_i^k - (\tilde s_{i-1}^k + \tau + \hat t_i)|
\leq |\tilde s_i^k - (\tilde T_i^k + \hat t_i)|
+ |\tilde T_i^k - (\tilde s_{i-1}^k + \tau)| \leq \tau + 4AR_0 + 2A\rho'.
\]
Recalling that $s_i = s_{i-1} + \tau + t_i$, we have
\begin{align*}
|\tilde s_{i}^k - s_i|
&\leq |\tilde s_i^k - (\tilde s_{i-1}^k + \tau + \hat t_i)|
+ |(\tilde s_{i-1}^k + \tau + \hat{t}_i) - (s_{i-1} + \tau + t_i)| \\
&\leq \tau + 4AR_0 + 2A\rho' + |\tilde s_{i-1}^k - s_{i-1} + \hat{t}_i - t_i| \\
&= \tau + 4AR_0 + 2A\rho' + |\tilde s_{i-1}^k - s_{i-1} + (s_{i-1} - \hat s_{i-1})| \\
&\leq \tau + 4AR_0 + 2A\rho' + A(R_0+2\rho'),
\end{align*}
where the last inequality uses Lemma \ref{lem:tilde-hat}.  One further application of Lemma \ref{lem:tilde-tilde} gives
\begin{align*}
|\tilde T_i^k - T_i| &\leq |\tilde T_i^k - (\tilde s_{i-1}^k + \tau)| + |(\tilde s_{i-1}^k + \tau) - (s_{i-1} + \tau)| \\
&\leq \tau + 2AR_0 + (\tau + 5AR_0 + 4A\rho'),
\end{align*}
which proves the lemma.
\end{proof}

\subsection{Equidistribution of closed geodesics}\label{sec:equidist}

To prove that the unique MME $\mu$ is the limiting distribution of (homotopy classes of) closed geodesics, we start by recalling the following standard fact, which can be proved by following \cite[Theorem 9.10]{pW82} or \cite[Proposition 6.4]{gK98}.

\begin{proposition}\label{prop:limidis}
Given $\JJ\subset (0,\infty)$, suppose that for each $T\in \JJ$ there is a set $\mathcal{C}(T)$ of closed geodesics with length at most $T$ such that
%Given a sequence $T_k\to\infty$, let $\mathcal{P}(T_k)$ be a set of closed geodesics with length at most $T_k$; write $P(T_k)$ for the cardinality of $\mathcal{P}(T_k)$. Suppose that
\begin{equation}\label{eqn:Tk-h}
\lim_{\substack{T \to \infty\\ T\in \JJ}} \frac 1{T} \log \card \mathcal{C}(T)= \htop(f_1) > 0,
\end{equation}
and that there is $\epsilon>0$ such that each $\mathcal{C}(T)$ is $(T,\epsilon)$-separated for each $T\in \JJ$, meaning that any two distinct elements of $\mathcal{C}(T)$ admit unit speed parametrizations $c_1,c_2$ such that $d(c_1(t),c_2(t)) \geq \epsilon$ for some $t\in [0,T]$. Then the invariant probability measures $\nu_T$ defined by
\begin{equation}\label{eqn:nuT}
\int \varphi \,d\nu_T = \frac 1{\card \mathcal{C}(T)} \sum_{c\in \mathcal{C}(T)} \int_0^{|c|} \varphi(c(t))\,dt
\end{equation}
have the property that every weak* accumulation point of $\{\nu_{T}\}_{T\in \JJ}$  is a measure of maximal entropy.
\end{proposition}

%Recall from \S\ref{sec:per-limit} that $\mathcal{P}(T)$ is a maximal collection of pairwise non-homotopic closed geodesics with period at most $T$, and $P(T)$ is its cardinality. 
Since the MME $\mu$ is unique, we could prove that $\mu$ is the limiting distribution of closed geodesics by applying Proposition \ref{prop:limidis} with $\mathcal{C}(T) = \mathcal{P}(T)$ (recall \S\ref{sec:per-limit}) if we knew that
\begin{enumerate}
\item $\lim_{T\to\infty} \frac 1T \log P(T) = \htop(f_1)$, and
\item there is $\epsilon>0$ such that $\mathcal{P}(T)$ is $(T,\epsilon)$-separated.
\end{enumerate}
Unfortunately the second of these turns out to be false in general, so we must do a little more work, but the failure is not dramatic. First we obtain the growth rate estimate by recalling the following special version of a result from \cite{CK02} which was partially based on ideas  in \cite{gK83}. 

\begin{theorem}\label{thm:specification}
Let $(M,g)$ be a closed Riemannian manifold (not necessarily without conjugate points) admitting a metric of negative curvature.
Let $P(T)$ be the number of free homotopy classes containing a closed geodesic of period less than $T$. Then there exist positive constants
$A,B$ and $T_0$ such that
\begin{equation}\label{eqn:APB}
A \frac{e^{hT}}{T} \le P(T) \le B e^{hT}
\end{equation}
for all $T \ge T_0$, where 
$
h = \lim_{r\to\infty}\frac{1}{r}\log \vol (B(p,r))
$
is exponential volume growth (volume entropy) on the universal covering $\wM$. 
\end{theorem}

This result was formulated  in \cite{CK02} in the general context of Gromov hyperbolic metric spaces
and since $\wM$ is Gromov hyperbolic the result applies in our setting.
As remarked in \eqref{eq:topent} in case of no conjugate points the volume entropy coincides with the topological entropy
of the geodesic flow, and we conclude that \eqref{eqn:APB} holds with $h=\htop(f_1)$.

Now we turn our attention to the question of whether orbits in $\mathcal{P}(T)$ are $(T,\epsilon)$-separated. In general, two closed geodesics of different lengths can lie in different free homotopy classes but still shadow each other arbitrarily closely: consider the center circle and boundary circle of a (flat) M\"obius strip. However, if the lengths are close enough, this cannot occur.

\begin{lemma}\label{lem:hom-sep}
Let $\epsilon>0$ be such that $2\epsilon$ is smaller than the injectivity radius of $M$, and let $c_1,c_2$ be closed geodesics with lengths in $(T-\epsilon,T]$ for some $T>0$. If $c_1,c_2$ are not homotopic, then there is $t\in [0,T]$ such that $d(c_1(t),c_2(t)) \geq \epsilon$.
\end{lemma}
\begin{proof}
Suppose that $d(c_1(t),c_2(t)) < \epsilon$ for all $t\in [0,T]$. We prove that $c_1,c_2$ are homotopic.
Let $|c_i|$ be the length of $c_i$, and let $\beta = |c_2|/|c_1|$; observe that $|\beta-1| \leq \epsilon/|c_1|$. Then for each $t\in [0,|c_1|]$ we have
\[
d(c_1(t),c_2(\beta t)) \leq d(c_1(t),c_2(t)) + |(\beta-1)t| < \epsilon + \frac{\epsilon}{|c_1|} t \leq 2\epsilon.
\]
This is smaller than the injectivity radius, so there is a continuous vector field $V$ on $c_1$ such that $\exp_{c_1(t)} V(c_1(t)) = c_2(\beta t)$ for all $t\in [0,|c_1|]$. Then the map $H\colon [0,1]\times [0,|c_1|] \to M$ given by $H(s,t) = \exp_{c_1(t)}(sV(c_1(t)))$ is a homotopy between $c_1$ and $c_2$.
\end{proof}

Motivated by Lemma \ref{lem:hom-sep}, let $P^*(T)$ denote the number of free homotopy classes of closed geodesics with lengths in $(T-\epsilon,T]$. (Note that by Lemma \ref{lem:same-length}, the length does not depend on the representative chosen.) A priori it is possible that $P^*(T)$ will be `too small' for some values of $T$,
%\footnote{In fact, since the present paper was completed, we have used the mixing property of $\mu$ and a Margulis-type argument to prove that $P^*(T)$ grows with the correct rate, so that one can apply Proposition \ref{prop:limidis} with $\JJ = (0,\infty)$ and obtain a stronger equidistribution result than the one here; see \cite[Theorem 1.3]{CKW}.  However, this requires a substantial amount of work that goes beyond the scope of this paper.} so we define
\[
\JJ := \{T>0 : P^*(T) \geq T^{-3} e^{hT}\}.
\]
%Let $U$ be an arbitrary convex weak*-neighborhood of $\mu$ in the space of invariant Borel probability measures. Proposition \ref{prop:limidis} and uniqueness of the MME yield $\bar T$ such that for all $T\in \JJ \cap (\bar T,\infty)$ and any set $\mathcal{C}(T)$ of pairwise non-homotopic closed geodesics with lengths in $(T-\epsilon,T]$, we have $\nu_T \in U$, where $\nu_T$ is defined by \eqref{eqn:nuT}.
Now given $T>0$, write $T_n := T-n\epsilon$ and $N = \lfloor T/\epsilon \rfloor$, and observe that $P(T) = \sum_{n=0}^{N} P^*(T_n)$. We split the sum into three parts: writing
\[
\II_T := \{ n : T_n \in \JJ \cap (T/2,\infty) \}
\quad\text{and}\quad
\II_T' := \{ n : T_n \in (T/2,\infty) \setminus \JJ \},
\]
we have
\begin{equation}\label{eqn:I}
P(T) = P(T/2) + \sum_{n\in \II_T} P^*(T_n)
+ \sum_{n\in \II_T'} P^*(T_n).
\end{equation}
From \eqref{eqn:APB} we see that
\[
P(T) \geq \frac AT e^{hT}
\quad\text{and}\quad
P(T/2) \leq B e^{hT/2},
\]
and thus
\begin{equation}\label{eqn:T2T}
\frac{P(T/2)}{P(T)} \leq \frac{B e^{hT/2}}{\frac AT e^{hT}} = \frac{TB}{A} e^{-hT/2} \to 0
\text{ as } T\to\infty.
\end{equation}
The definition of $\JJ$ gives
%\begin{equation}\label{eqn:I2}
\[
\sum_{n\in \II_T'} P^*(T_n) < \sum_{n\in \II_T'} T_n^{-3} e^{hT_n}
\leq \frac T{2\epsilon} \Big( \frac T2\Big)^{-3} e^{hT} = \frac{4}{\epsilon T^2} e^{hT},
%\end{equation}
\]
and thus
\begin{equation}\label{eqn:I'}
\frac{\sum_{n\in \II_T'} P^*(T_n)}{P(T)} \leq \frac 4{\epsilon T^2} e^{hT} \frac TA e^{-hT}
= \frac 4{\epsilon AT} \to 0 \text{ as } T\to\infty.
\end{equation}
Now given a set $\mathcal{P}(T)$ of pairwise non-homotopic closed geodesics with lengths at most $T$, let $\mu_T$ be the corresponding periodic orbit measures defined in \eqref{eqn:muT}, and let $\nu_T$ be the periodic orbit measures associated to $\mathcal{P}^*(T) := \{ c\in \mathcal{P} : |c| \in (T-\epsilon,T]\}$. Observe that
\begin{equation}\label{eqn:mut-nut}
\begin{aligned}
\mu_T &= \frac 1{P(T)} \sum_{n=0}^N P^*(T_n) \nu_{T_n} \\
&= \frac {P(T/2)}{P(T)} \mu_{T/2} + \frac{\sum_{n\in \II_T} P^*(T_n) \nu_{T_n}}{P(T)}
+ \frac{\sum_{n\in \II_T'} P^*(T_n) \nu_{T_n}}{P(T)}.
\end{aligned}
\end{equation}
As $T\to\infty$, the total weight of the first expression goes to $0$ by \eqref{eqn:T2T}, and the total weight of the third expression goes to $0$ by \eqref{eqn:I'}. It follows that the limit points of $\{\mu_T\}$ as $T\to\infty$ are the same as the limit points of $\{\nu_T\}_{T\in\JJ}$. By Proposition \ref{prop:limidis} and Lemma \ref{lem:hom-sep}, every such limit point is a measure of maximal entropy. Because the measure of maximal entropy is unique, we conclude that $\mu_T\to \mu$, which completes the proof of the equidistribution property (Definition \ref{def:periodics}) claimed in Theorem \ref{thm:higher-dim}.

\section{Patterson--Sullivan measure and the MME}\label{sec:PS}
In this section we assume that $M$ is a closed Riemannian manifold without conjugate points having the divergence property of geodesic rays and admitting a metric of negative sectional curvature, i.e., we are only assuming that  conditions \ref{H1} and \ref{H2} in Definition \ref{def:H} are satisfied.
We will show that under this assumption the Patterson--Sullivan measure can be used to define a measure of maximal entropy
which is fully supported on $T^1M$.  If we add the conditions  \ref{H3} and \ref{H4} from Definition \ref{def:H} we obtain uniqueness 
as was shown in \S\ref{sec:proof}.

\subsection{Poincar\'{e} series and the Patterson--Sullivan measure}
 
If $\Gamma$ denotes the group of deck transformations, for $p,q \in  \wM$ and $s \in \mathbb{R}$, we consider the Poincar\'{e} series
 \[
P(s,p,q) = \sum\limits_{\gamma \in \Gamma} e^{- sd(p,\gamma q)}.
\]
Since  $\wM$ is Gromov hyperbolic  it follows from \cite {mCo93}  that the series
converges for $s> h$ and diverges for $s \le h$, where $h$ is the topological entropy.
For $x \in \wM$ the set 
$\Lambda(\Gamma)$ of accumulation points of the orbit
 $\Gamma x$ in $ \wM$ is called the limit set. Since $\wM$ is cocompact we have $\Lambda(\Gamma) =
\partial \wM $. 
Fix $x \in \wM $, $s > h$ and consider for each $p \in \wM$  the measure
\begin{equation}\label{eqn:nupxs}
\nu_{p,x,s} = \frac1{P(s,x,x)} \sum\limits_{\gamma \in \Gamma} e^{-sd(p,\gamma x)} \delta_{\gamma x}
\end{equation}
where $\delta_y$ is the Dirac mass associated to $y \in \wM $.
Using the fact that
$e^{-s d(p,x)} e^{-s d(x,\gamma x)} \leq e^{-s d(p,\gamma x)} \leq e^{sd(p,x)} e^{-sd(x,\gamma x)}$ for every $x,p\in \wM$ and $\gamma\in \Gamma$, we see that
\begin{equation}\label{eqn:nu-weight}
e^{-sd(p,x)} \le \nu_{p,x,s} (\cl( \wM) ) \le e^{sd(p,x)};
\end{equation}
in particular, the $\nu_{p,x,s}$ are all finite.  Moreover, we clearly have
\begin{equation}\label{eqn:supp-nu}
\Gamma x \subset \supp \nu_{p,x,s} \subset \overline{\Gamma
x}.
\end{equation}
Now choose for a fixed $p \in \wM  $ and a weak limit $\lim\limits_{k\to \infty} \nu_{p,x,s_k} = : \nu_p$.

The divergence of the series $P(s,x,x)$ for $s = h$ and the
discreteness of $\Gamma$ yields that the support of
$\nu_{p}$ is contained in the limit set. Moreover, one obtains:

\begin{proposition}\label{prop:PS} There is a sequence \(s_k\to h\) as \(k\to\infty\) such that 
for every $p \in \wM$ the weak* limit $\lim\limits_{k\to \infty} \nu_{p,x,s_k} =:
\nu_p$ exists.
The  family of measures $\{\nu_p\}_{p \in \wM}$ has the following 
properties.
\begin{enumerate}[label=\upshape{(\alph{*})}]
\item\label{PS-a} $\{\nu_p\}_{p \in \wM}$ is $\Gamma$-equivariant: for all Borel sets $A \subset \partial \wM$, we have
\[
\nu_{\gamma p} (\gamma A) = \nu_p(A).
\]
\item\label{PS-c} $\frac{d \nu_q}{d\nu_p} (\xi) = e^{-h b_p(q, \xi)}$
for almost all $\xi  \in \partial \wM$, where $b_p(q,\xi)$ is as in Definition \ref{def:busemann}.
\item\label{PS-b} $\supp \ \nu_p = \partial \wM$ for all $p \in
\wM$.
\end{enumerate}
\end{proposition}
\begin{proof} 
Fix $p\in \wM$. Let $s_k\to h$ be such that the weak* limit $\lim_{k\to \infty} \nu_{p,x,s_k} =: \nu_p$ exists.  To see that $s_k$ works for all $q$, define a function $\psi\colon \cl(\wM) \to \RR$ by
\[
\psi(z) = \begin{cases}
d(q,z) - d(p,z) & \text{if } z\in \wM, \\
b_p(q,z) &\text{if } z\in \ideal,
\end{cases}
\]
and observe that \eqref{eqn:nupxs} gives $\frac{ d\nu_{q,x,s}}{d\nu_{p,x,s}} = e^{-s\psi}$ for all $q$ and $s$. The function $\psi$ is continuous by \eqref{eqn:bpqxi}, so using the fact that $e^{-s_k \psi} \to e^{-h\psi}$ uniformly, we deduce that for any continuous $\phi\colon \cl(\wM)\to\RR$ we have
\[
\lim_{k\to\infty} \int \phi\,d\nu_{q,x,s_k}
= \lim_{k\to\infty} \int \phi e^{-s_k\psi} \,d\nu_{p,x,s_k} = \int \phi e^{-h \psi} \,d\nu_p.
\]
This proves that $\nu_q := \lim_k \nu_{q,x,s_k}$ exists for all $q\in \wM$, and that \ref{PS-c} holds.

For \ref{PS-a}, it suffices to observe that \eqref{eqn:nupxs} gives
\[
\nu_{\gamma p,x,s}(\gamma A)
= \frac 1{P(s,x,x)} \sum_{\alpha\in \Gamma} e^{-sd(\gamma p,\alpha x)} \delta_{\alpha x}(\gamma A),
\]
and that upon re-indexing by $\beta = \gamma^{-1}\alpha$, the sum is equal to
\[
\sum_{\beta\in \Gamma} e^{-sd(\gamma p, \gamma \beta x)} \delta_{\gamma \beta x}(\gamma A)
= \sum_{\beta\in \Gamma} e^{-sd(p, \beta x)} \delta_{\beta x}(A),
\]
so that $\nu_{\gamma p,x,s}(\gamma A) = \nu_{p,x,s}(A)$. Taking a limit as $s_k\to h$ gives \ref{PS-a}.

%We first observe that given a compact subset \(K\subset\wM\), there is a sequence \(s_k\to h\) as \(k\to\infty\) such that for every $p \in K$ the weak limit $\lim\limits_{k\to \infty} \nu_{p,x,s_k} =:\nu_p$ exists. On the other hand, the definition of $\nu_{p,x,s}$ immediately gives$\nu_{\gamma p,x,s} (\gamma A) = \nu_{p,x,s} (A)$ for all $p,x,s$.  Thus, since \(M\) is a compact quotient we  have the same sequence \(s_k\to h\) as \(k\to\infty\) such that for every $p \in \wM$ the weak limit $\lim\limits_{k\to \infty} \nu_{p,x,s_k} =:\nu_p$ exists. This gives the first assertion.

%The second  assertion follows immediately from \eqref{eqn:bpqxi}.
%The third assertion follows from 
Finally, for \ref{PS-b}, we use
the fact that $\Gamma$ acts minimally on $\partial  \wM$ (see Remark \ref{rem:minmal}).
Suppose
$\xi \notin \supp \nu_p$ then $\gamma \xi  \notin \supp \nu_p$ for all $\gamma \in \Gamma$, since given an open neighborhood  
$U \subset   \wM$ of $\xi$ with $\nu_p(U) =0$ we have  $0 = \nu_{\gamma p}(\gamma(U)) =  \nu_p(\gamma(U))$ 
by \ref{PS-a} and \ref{PS-c}. Since the orbit  $\Gamma \xi$ is a dense and the $\supp \nu_p$ is closed we obtain $\supp \nu_p = \emptyset$ 
which contradicts the non-triviality
of the Patterson-Sullivan measure.
\end{proof} 
\begin{remark}\label{rem:PS}
Since  $|b_p(q,\xi)| \leq d(p,q)$ for all $p,q,\xi$ property  \ref{PS-c}  implies that for every $p,q\in \wM$ and any measurable subsets $A\subset \ideal$ that
 $\frac{\nu_p(A)}{\nu_q(A)} \leq e^{hd(p,q)}$.
\end{remark}
For $\xi \in \partial  \wM$ and $p \in \wM$
consider the projections
\[
pr_\xi \colon  \wM \to \partial  \wM
\quad\text{and}\quad
pr_p \colon  \wM \setminus \{p\} \to  \partial  \wM
\]
along geodesics emanating from $\xi$ and $p$, respectively. That is, $pr_\xi(x)
= c_{\xi, x}(\infty)$, where $c_{\xi,x}$ is the geodesic with
$c_{\xi,x}(-\infty) = \xi$,
 $c_{\xi, x}(0) = x$ and $pr_p(x) =
c_{p,x}(\infty)$, where $c_{p,x}$ is the geodesic
 with $c_{p,x}(0) = p$,
$c_{p,x}(d(p,x))=x$. 

\begin{lemma}\label{lem:open-shadow}
There exists $R>0$ such that
for all $x \in  \cl(\wM) $ and $p \in \wM$, the \textit{shadow set} 
$pr_x B(p,R)$  of the open geodesic ball $B(p,R)$ with center $p$ and radius $R$ contains an open set in $\partial  \wM$.
\end{lemma}
\begin{proof}
For \(x\in\cl{(\wM)}\) and \(p\in\wM\), let $v = -V(p,x)$ be given by Proposition \ref{prop:Eb1}. By the definition of the topology on \(\partial \wM\),  for every  \(v\in T_p^1\wM\) and \(\varepsilon>0\) we have \(A_{\varepsilon}(v):=\{c_w(\infty): \angle_p(v,w)<\varepsilon\}\) is open in \(\partial\wM\). For every \(\eta\in A_\varepsilon(v)\) there exists a unique geodesic  \(c^0_{\xi,\eta}\) with respect to the metric of negative curvature joining \(\xi:=c_v(-\infty)\) and \(\eta\); every such geodesic stays at a bounded distance to \(p\) and this distance can be made arbitrary small by choosing \(\varepsilon\) arbitrary small. By the Morse Lemma, every geodesic \(c^0_{\xi,\eta}\) corresponds to at least one geodesic \(c_{\xi,\eta}\) which stay at \(R_0\) distance (see Theorem \ref{thm:Morse}). By choosing \(\varepsilon\) small enough and \(R\) large, we can guarantee  \(d(p,c_{\xi,\eta})<R\) which implies that \(A_{\varepsilon}(v)\subset pr_x B(p,R)\).
\end{proof}

Using Lemma \ref{lem:open-shadow} we obtain:
\begin{proposition}\label{prop:PS-bds}
Let $\{\nu_p\}_{p\in  \wM }$ be the Patterson--Sullivan measures and fix $\rho\geq R$, where $R$ is as in Lemma \ref{lem:open-shadow}.
\begin{enumerate}[label=\upshape{(\alph{*})}]
\item\label{bds-a}
There exists $\ell=\ell(\rho) > 0$ such that for every $x\in \cl(\wM)$, we have
\[
\nu_p(pr_x B(p,\rho)) \ge \ell.
\]
\item\label{bds-b}
There is a constant $b = b(\rho)$
such that for all $x \in \wM$ and $\xi = c_{p,x}(-\infty)$,
\[
\frac{1}{b} e^{-h d(p,x)} \le \nu_p(pr_\xi(B(x,\rho)) \le be^{- h d(p,x)}.
\]
\item\label{bds-c}
A similar estimate holds if we project from $p \in \wM$, namely there is a constant
$a = a(\rho) > 0$ such that for all $p\in \wM$,
\[
\frac{1}{a} e^{-h d(p,x)} \le \nu_p(pr_p(B(x,\rho))) \le a e^{-h
d(p,x)}.
\]
\end{enumerate}
\end{proposition}
\begin{proof}
The last two estimates follow from \ref{bds-a} and the defining properties of $ \nu_p$.  To see this, observe that 
given $A \subset \partial \wM$, Proposition \ref{prop:PS}\ref{PS-c} gives
\[
\nu_p(A) = \int\limits_{A} e^{-h b_x(p, \eta)} d\nu_x(\eta).
\]
If $A = pr_\xi B (x, \rho)$ or $pr_p B(x, \rho)$ then corollary \ref{cor:c1} implies that
 $|(b_x (p, \eta) - d(p,x))|$ is bounded by
a constant for all $\eta  \in A$, which yields \ref{bds-b} and \ref{bds-c}.

The first estimate is a consequence of the following steps.\\
\underline{Step 1:} \hspace{0,5cm}
{\it $\supp \;\nu_p = \partial \wM$ for one and, hence, for all $p \in
\wM$} using Proposition \ref{prop:PS}\ref{PS-b}. \\

For $v \in T^1_x\wM = \pi^{-1}(x)$ and $\epsilon>0$,
let $A_\epsilon(v) = \{c_w(\infty) : w\in T_x^1\wM$ and $\angle(v,w) < \epsilon\} \subset \ideal$ as in the proof of Lemma \ref{lem:open-shadow}.
% consider the open neighborhood
%\[
%C_\epsilon(v) = \{c_w(\infty)\; | \;w \in T_x^1 \wM  \; \; \text{ and } \; \; \angle (v,w)
%< \epsilon\} \subset \ideal.
%\]
Fix a compact set $K\subset\wM$ such that $ \bigcup_{\gamma \in \Gamma} \gamma(K) =  \wM $ and a reference
point $x_0 \in K$.
Then it follows:\\
\underline{Step 2:} \hspace{0,5cm}  {\it For all $\rho \ge R$ there exists $\epsilon > 0$ such that for all
$p \in K$ and $x \in    \cl( \wM) $
\[
 A_\epsilon(v) \subset pr_x (B(p , \rho )) \;
\]
for some $v \in T^1_{x_0}\wM$.}\\
Suppose Step~2 is false. Then there exists $\rho \ge R$ and sequences $p_n \in K$,
$x_n \in  \cl( \wM)$  such that
\[
A_{1/n} (v) \not\subseteq pr_{x_n}(B(p_n, \rho))
\]
for all $v \in T^1_{x_0}\wM$.  We can assume after choosing a subsequence
that $x_n \to  \xi \in \cl( \wM)$ and $p_n \to p \in K$. Since  $pr_{\xi}(B(p, \rho))$ contains some open set in $\partial \wM$
there exists $\epsilon >0$ and $v_0 \in T^1_{x_0}\wM$ such that $A_{\epsilon}(v_0) \subset
pr_\xi (B(p, \rho))$. The continuity of the projection implies the existence
of $n_0$ such that for all $n \ge n_0$ we have: $CA_{\epsilon/2}(v_0) \subset
pr_{x_n} (B(p_n, \rho))$. But this contradicts the choice of the sequence. Then Step 2 is true.\\
\underline{Step 3:} \hspace{0,5cm}  {\it
For all $\epsilon > 0$ there exists a constant $\ell= \ell(\epsilon) > 0$ such that
\[
\nu_p (A_\epsilon(v)) > \ell
\]
for all $v \in T^1_{x_0}\wM $ and $p \in K$.}
This is a consequence of the following facts: each $\nu_p$ is fully supported (Step 1); $\sup \{b_p(q,\xi) : \xi\in \ideal, p,q\in K\} < \infty$ by compactness and continuity; and there is a finite collection of open sets in $\ideal$ such that each $A_\epsilon(v)$ contains an element of this collection.

Now consider $x \in  \cl( \wM)$ and $p \in \wM$. Choose $\gamma \in \Gamma$ such
that
 $\gamma p \in K$. Since $\nu_p(pr_x(B(p, \rho)) = \nu_{\gamma p}
(pr_{\gamma x}
 B(\gamma p, \rho))$ the estimate (a) follows from Steps 2 and
3.
 \end{proof}
 
 \subsection{Construction of the measure of maximal entropy using the Patterson-Sullivan measure}
Now we construct an
invariant measure for the geodesic flow using the Patterson-Sullivan measures $\nu_p$. Broadly speaking, we follow the approach in \cite{gK98}, which was originally carried out in negative curvature in \cite{vK90}; however, as we will see below, the present setting introduces some technical difficulties that require some work to overcome.

By Proposition \ref{prop:PS}\ref{PS-c},  $\nu_p$ is
 $\Gamma$-quasi-invariant with Radon-Nikodym cocycle
\begin{equation}\label{eqn:RN}
f(\gamma, \xi) = e^{-h b_p(\gamma^{-1} p, \xi)} = \frac{d\nu_{\gamma^{-1} p}}{d\nu_p}(\xi).
\end{equation}
For $(\xi, \eta) \in \sqbd := \dididi$ consider
\begin{equation}\label{eqn:beta-p}
\beta_p(\xi, \eta) = - (b_p(q, \xi) + b_p(q, \eta)) \;,
\end{equation}
where $q$ is a point on a geodesic $c$ connecting $\xi$ and $\eta$.
In geometrical terms $\beta_p(\xi, \eta)$ is the length of the
segment $c$ which is cut out by the horoballs through
$(p, \xi)$ and $(p, \eta)$.
Since $\grad_q b_p(q,\xi) =- \grad_q b_p(q,\eta)$ for all points on geodesics
connecting $\xi$ and $\eta$, this number is independent of the choice of
$q$. An easy computation using \eqref{eqn:RN}, see \cite[Lemma 2.4]{gK98}, shows:
\begin{lemma}\label{5.4.A}
For $p \in \wM$, the measure $\bar\mu$ on $\sqbd$ defined by
\[
d \bar\mu (\xi, \eta) = e^{h\beta_p(\xi,\eta)}
 d\nu_p(\xi) d\nu_p(\eta)
\]
%defines a $\Gamma$-invariant measure on $\partial \wM \times \partial \wM \setminus\diag$.
is $\Gamma$-invariant.
\end{lemma}

Now we use $\bar\mu$ to produce a $\Gamma$-invariant and flow-invariant Borel measure $\tmu$ on $T^1\wM$ that projects to a finite flow-invariant Borel measure $\mu$ on $T^1 M$. We will need the projection $P\colon T^1\wM  \to  \sqbd$  given by $P(v) = (c_v(-\infty), c_v(\infty))$, where $c_v$ is the geodesic with $\dot c_v(0)=v$. 

In negative curvature, we can proceed as in \cite{vK90}: $P^{-1}(\xi,\eta)$ is a single trajectory -- the set of tangent vectors to a single geodesic -- and so writing $\lambda_{\xi,\eta}$ for Lebesgue measure on $P^{-1}(\xi,\eta)$, one obtains a $\Gamma$-invariant and flow-invariant measure on $T^1\wM$ by
\begin{equation}\label{eqn:Kaim}
\tmu(A) = \int_{\sqbd} 
\lambda_{\xi,\eta}(A) \,d\bar\mu(\xi,\eta).
\end{equation}
One can follow the same approach in nonpositive curvature, where $P^{-1}(\xi,\eta)$ is either a single geodesic or a flat totally geodesic submanifold of $\wM$ on which the flow acts isometrically \cite{gK98}.
In our setting, however, the flow need not act isometrically on $P^{-1}(\xi,\eta)$ (the flat strip theorem fails), and on such sets it is not clear how to define a flow-invariant measure in a measurable and $\Gamma$-invariant way. Nevertheless, we can prove the following.

\begin{theorem}\label{5.4.C}
Let $M$ be a smooth closed Riemannian manifold without conjugate points satisfying conditions \ref{H1}--\ref{H4}. Then $P^{-1}(\xi,\eta)$ is a single geodesic for $\bar\mu$-a.e.\ $(\xi,\eta) \in \sqbd$, and thus \eqref{eqn:Kaim} defines a $\sigma$-finite Borel measure $\tmu$ on $T^1\wM$. This measure is fully supported, gives full weight to the expansive set $\mathcal{E}$ from \eqref{eqn:E}, and is the lift of the unique MME $\mu$ on $T^1M$ as in \eqref{eqn:tmu-lift}.
In particular, $\mu$ is ergodic fully supported on $T^1M$.
%the measure $\mu$ constructed by the above procedure is a measure of maximal entropy for the geodesic flow on $T^1M$: $h_{\mu} (f_1) = \htop(f_1)  = h$.
%\gk{The condition (H4) implies that the expansive set has full measure with respect to  $\mu$. The reason is that $\Gamma$ acts ergodically  on $\dididi$ with respect to $\bar \mu$ and therefore the expansive set on  $\dididi$ has zero
%or full measure. If the measure is zero $\mu$ provides a measure of maximal entropy such that the expansive set on  $T^1\wM$ has measure zero contradicting (H4). \\
%If this is true I would suggest to write: If the metric  satisfies conditions \ref{H1} to \ref{H4} the unique MME is fully supported on the expansive set.
%This is also part of the claim in Theorem 1.2. }
%If in addition we know that the measure of maximal entropy is unique and gives full weight to $\pr_*\mathcal{E}$, where $\mathcal{E}$ is the set of expansive vectors from \eqref{eqn:E}, then $\bar\mu$ gives full weight to the set of pairs $(\xi,\eta)$ for which $P^{-1}(\xi,\eta)$ is a single trajectory, and $\tmu$ is given by \eqref{eqn:Kaim}. In this case the unique MME $\mu$ is fully supported.
\end{theorem}

%In order to define a $\Gamma$-invariant and flow-invariant measure via this approach, one needs to produce a family of measures $\lambda_{\xi,\eta}$ such that
%\begin{itemize}
%\item for every Borel $A\subset T^1\wM$, the function $\sqbd \to \RR$ defined by $(\xi,\eta) \mapsto \lambda_{\xi,\eta}(A)$ is $\bar\mu$-measurable;
%\item each $\lambda_{\xi,\eta}$ is supported on $P^{-1}(\xi,\eta)$ and is flow-invariant;
%\item the $\Gamma$-equivariance property $\gamma_* \lambda_{\xi,\eta} = \lambda_{\gamma\xi,\gamma\eta}$ holds.
%\end{itemize}
%One idea at this point would be to select for each $(\xi,\eta)\in\sqbd$ a single trajectory $V(\xi,\eta) \subset T^1\wM$ corresponding to one of the geodesics connecting $\xi$ and $\eta$, and then take $\lambda_{\xi,\eta}$ to be Lebesgue measure along this geodesic. However, it is not clear how to define $V$ in a way that is simultaneously measurable and $\Gamma$-equivariant, so we proceed slightly differently, by weakening the requirement of flow-invariance slightly.

The rest of this section is devoted to proving Theorem \ref{5.4.C}. Although we will ultimately conclude that $P^{-1}(\xi,\eta)$ is a single trajectory $\bar\mu$-a.e., this will not come until the end of the proof: first we must construct an MME $\mu$ using $\bar\mu$ \emph{without} knowing this fact, and then use \ref{H4} to deduce that the measure $\tmu$ given by \eqref{eqn:tmu-lift} gives full weight to $\mathcal{E}$, at which point the construction of $\mu$ will finally allow us to deduce the desired result for $\bar\mu$.

\begin{remark}\label{rmk:H1H2}
As we will see in the proof, Theorem \ref{5.4.C} remains true if we replace \ref{H3} and \ref{H4} with the assumption that there is a unique MME $\mu$ and that the lift $\tmu$ defined by \eqref{eqn:tmu-lift} satisfies $\tmu(\mathcal{E}^c)=0$. The construction below produces an MME even if we only assume that $M$ is a manifold without conjugate points satisfying \ref{H1} and \ref{H2}. The extra assumptions are not used until we deduce the expansivity-related properties of this MME, including \eqref{eqn:Kaim} and full support.
\end{remark}

To prove Theorem \ref{5.4.C}, most of the work goes into producing an MME on $T^1M$ using $\bar\mu$. First define an equivalence relation on $T^1\wM$ by writing 
\begin{equation}\label{eqn:sim}
\text{$v\sim w$ iff $H^s(v) = H^s(w)$ and $H^u(v) = H^u(w)$.}
\end{equation}
Write $[v]$ for the equivalence class of $v$, which projects injectively under $\pi$ to the compact set $H^s(v) \cap H^u(v)$.

\begin{lemma}\label{lem:sim-pr}
If $v,w\in T^1\wM$ are such that $v\sim w$ and $\pr_* v = \pr_* w$, then  $v=w$.
\end{lemma}
\begin{proof}
Given $v,w$ as in the hypothesis,  it follows that  $\gamma_* v =w$ for some  $\gamma\in\Gamma$. Suppose that $v\neq w$; then $\gamma$ is not the identity, so by \cite[\S1.4]{wK71}, 
$\gamma$ fixes exactly two points on $\ideal$, which are the endpoints of an axis $c$. In other words, there  
exist a geodesic $c\colon \RR\to \wM$ and a real number $a\neq 0$ such that 
$\gamma c(t) = c(t+a)$ for all $t\in\RR$, and such that $c(\pm\infty)$ are the only two fixed points of $\gamma$ in $\ideal$.

Now observe that $v\sim w$ gives $P(v) = P(w) = P(\gamma_* v) = \gamma P(v)$, so $P(v) = P(w) = c(\pm\infty)$ since $\gamma$ has no other fixed points. Without loss of generality assume that $c(\infty) = c_v(\infty) = c_w(\infty)$ and that $c(0) \in H^s(v)$, so that $H^s(w) = H^s(v) = H^s(\dot{c}(0))$. Then we have
\begin{align*}
H^s(\dot{c}(0)) &= H^s(w) = H^s(\gamma_* v) 
= \gamma H^s(v) \\
&= \gamma H^s(\dot{c}(0))
= H^s(\gamma_* \dot{c}(0))
= H^s(\dot{c}(a)),
\end{align*}
implying that $a=0$, so $\gamma$ is the identity. This contradicts our assumption that $v\neq w$, and proves the lemma.
\end{proof}

This equivalence relation projects to $T^1M$: we write $v\sim w$ if $v,w$ have lifts that satisfy \eqref{eqn:sim}.
Let $\tilde Q \colon T^1\wM \to T^1\wM/{\sim}$ and $Q \colon T^1M \to T^1M/{\sim}$ be the quotient maps.  These are continuous when we equip the quotient spaces with the metric $d([v],[w]) = \min \{d(v',w') : v'\in [v], w'\in [w]\}$. 
The flow $F$ takes equivalence classes to equivalence classes, $f_t[v] = [f_t v]$, and thus it descends to a continuous flow on the quotient spaces. 

Since by the Morse Lemma (Theorem \ref{thm:Morse}) equivalence classes are compact, the measurable selection theorem of Kuratowski and Ryll-Nardzewski \cite[\S5.2]{sS98} guarantees existence of a Borel measurable map $V \colon T^1M/{\sim} \to T^1M$ such that $Q\circ V$ is the identity. 
Then Lemma \ref{lem:sim-pr} guarantees that for every $v\in T^1\wM$,  $\pr_*^{-1}(V(\pr_*[v]))$ intersects $[v]$ in a single point, which we denote $\tilde V([v])$. We conclude that $\tilde V \colon T^1\wM/{\sim} \to T^1\wM$ is a measurable map such that $\tilde Q \circ \tilde V$ is the identity, and moreover
\begin{equation}\label{eqn:G-ev}
\tilde V(\gamma_* [v]) = \gamma_* \tilde V([v]).
\end{equation}
Now define a measure $\nu_{\xi,\eta}$ on each $P^{-1}(\xi,\eta)$ by fixing any $v\in P^{-1}(\xi,\eta)$ and putting for a Borel measurable set $A \subset T^1\wM$ 
\begin{equation}\label{eqn:nu-xi-eta}
\nu_{\xi,\eta}(A) = \Leb \{ t \in \RR : \tilde V([f_t v]) \in A \}.
\end{equation}
Note that this is independent of the choice of $v$.
Use this to define a measure $\tilde\nu$ on $T^1\wM$ by 
\[
\tilde\nu(A) = \int_{\sqbd} 
\nu_{\xi,\eta}(A) \,d\bar\mu(\xi,\eta).
\]
Observe that $\tilde\nu$ is $\Gamma$-invariant by \eqref{eqn:G-ev}, and as in \cite{vK90,gK98} it descends to a finite Borel measure $\nu$ on $T^1M$. Without loss of generality we scale the metric so that $\nu(T^1M)=1$.

The measure $\nu$ is not necessarily flow-invariant. However, the measure $m = Q_* \nu$ is a flow-invariant measure on $T^1M/{\sim}$ because $Q_* \nu_{\xi,\eta}$ is flow-invariant on each $P^{-1}(\xi,\eta)/{\sim}$.
Now the set $Q_*^{-1}(m)$ of Borel probability measures is weak* compact and closed under $(f_t)_*$ for all $t\in \RR$, so the usual argument from the Krylov--Bogolyubov theorem for producing an invariant probability measure (take a weak* limit point of the family $\nu_T = \frac 1T \int_0^T (f_t)_* \nu \,dt$ as $T\to\infty$) shows that there is a flow-invariant Borel probability measure $\mu$ on $T^1M$ with $Q_* \mu = m$. This lifts to a $\Gamma$-invariant and flow-invariant Borel measure $\tmu$ on $T^1\wM$ by \eqref{eqn:tmu-lift}.

We claim that $\mu$ is a measure of maximal entropy. For this we will need some estimates on $\tilde\nu$ that carry through to $\tmu$. More specifically: it follows from \eqref{eqn:nu-xi-eta} that $\nu_{\xi,\eta}(A) \leq \diam A$ for all $A$, and thus
\begin{equation}\label{eqn:tnu}
\tilde\nu(A) \leq \bar\mu(P(A)) \diam A.
\end{equation}
The same bound holds for each $(f_t)_*\tilde\nu$, and since $\tmu$ is a limit of convex combinations of such measures, we obtain the same bound for $\tmu$:
\begin{equation}\label{eqn:tmu-leq}
\tmu(A) \leq \bar\mu(P(A)) \diam A.
\end{equation}
To show that $\mu$ is a measure of maximal entropy, we consider a measurable partition
${\mathcal A} = \{A_1, \ldots , A_m\}$  of
$T^1M$  such that the diameters of all elements in ${\mathcal A}$ are less than $\epsilon$
with respect to the metric $d_1$ defined at the beginning of \S\ref{sec:geodesic}.

\begin{lemma}\label{5.4.B}
Let $0 < \epsilon < \min \{R, \mathrm{inj} (M) \}$, where $\mathrm{inj}(M)$
is the injectivity radius
 of $M$. Then there is a constant $a > 0$ such that
\[
\mu(\alpha) \le e^{-hn} a
\]
for all $n\in\NN$ and $\alpha \in {\mathcal A}_f^{(n)}$.
\end{lemma}
\begin{proof}
With $\epsilon$ fixed as in the hypothesis, let $R_2>0$ be the constant given 
by Lemma \ref{lem:endpts-suffice} with $R_1 = \epsilon$, and let $r_0 := \diam M$. We will determine the constant $a$ in terms of $\epsilon$, $R_2$, and $r_0$.

Fix $v\in \alpha$ and observe that $\alpha \subset \bigcap_{k=0}^{n-1} f_{-k} B_{d_1} (f_k v, \epsilon)$, so for every $w \in \alpha$ and $t\in [0,n]$ we have
$d(c_v(t), c_w(t)) \le \epsilon$. Let $p\in \wM$ be
the reference point used in the definition of the measure $\tilde \mu$ and
$\tilde v \in T^1\wM$ be a lift of $v$ such that $d(\pi \tilde v, p ) \le
\diam M = r_0$. 
Since $\epsilon < \mathrm{inj}(M)$ we can lift the
set $\alpha$ to a set $\tilde{\alpha} \in T^1\wM$ such
 that for all $\tilde{w}
\in \tilde\alpha$
 we have $d(c_{\tilde w}(t)$, $c_{\tilde v}(t)) \le
\epsilon$
 for all $t \in [0,n]$.

Let $c_{\tilde v}(n) = x $ and $\xi = c_{\tilde w}(- \infty)$. Let $c_{\xi, x}$
be the geodesic connecting $\xi$ and $x$ such that
$c_{\xi, x}(n) = x$. 
The construction of this geodesic in the proof of Lemma \ref{lem:ct} yields the estimate $d(c_{\xi,x}(t), c_{\tilde w}(t)) \leq R_2$ for all $t\in (-\infty,n]$. Applying this with $t=0$ we get
\[
d(c_{\xi, x}(0), p) \le d(c_{\xi, x}(0), \pi (\tilde w)) + d(\pi (\tilde w), p) \le
\epsilon + r_0 + R_2 =: r_1
\]
i.e., $\xi \in pr_x(B(p, r_1))$.
Therefore, if
$P\colon T^1\wM \to \sqbd$ denotes the
endpoint projection as in the paragraph following Lemma \ref{5.4.A}, we have
\begin{equation}\label{eqn:P-alpha}
P(\tilde \alpha) \subset \bigcup\limits_{\eta \in pr_x(B(p, r_1))}
\{\eta\} \times pr_\eta(B(x,\epsilon)).
\end{equation}
For each $\eta \in pr_x(B(p, r_1))$ choose a point
$q \in B(p, r_1)$ that lies on the geodesic $c_{\eta,x}$.
Then, using the transformation rule for the Patterson--Sullivan measure,
Proposition \ref{prop:PS-bds}\ref{bds-b} and the estimate
\[
d(q,x) \ge d(x, \pi \tilde v) - d(\pi \tilde v, p) - d(p,q) \ge n - r_0 - r_1,
\]
we obtain
\[
\nu_p(pr_\eta (B(x,\epsilon))) \le  e^{h d(p, q)} \nu_q(pr_\eta (B(x, \epsilon)))
\le e^{hr_1} b \; e^{-h d(q, x)} \le \tilde b e^{-hn}
\]
for a constant $\tilde b =b e^{h(r_0 + 2r_1)}>0$, where the first inequality follows from Remark \ref{rem:PS}.
Recalling the definition of $\bar\mu$ in Lemma \ref{5.4.A}, we see that
\begin{equation}\label{eqn:barmu-leq}
\bar\mu(P(\tilde\alpha)) \leq \bigg( \sup_{(\xi,\eta) \in P(\alpha)} e^{h\beta_p(\xi,\eta)} \bigg)
\tilde b e^{-hn} \nu_p(\ideal).
\end{equation}
The supremum is at most $\sup \{ e^{h\beta_p(\xi,\eta)} :$ the geodesic joining $\xi$ and $\eta$ intersects $B(p,r_1) \}$, which is finite. Thus combining \eqref{eqn:tmu-leq} and \eqref{eqn:barmu-leq} proves the lemma.
\end{proof}

Nowe we can complete the proof of Theorem \ref{5.4.C}. First we show that $\mu$ is an MME.  Choose a partition ${\mathcal A}$ as above, and use Lemma \ref{5.4.B} to deduce that
\[
H({\mathcal A}_{f_1}^{(n)}) = \sum\limits_{\alpha\in{\mathcal A}_{f_1}^{(n)}} \mu(\alpha) \;
(-\log \mu(\alpha))\ge (hn - \log a) \sum\limits_{\alpha\in{\mathcal A}_{f_1}^{(n)}}  \mu (\alpha) =
hn - \log a.
\]
Hence, $h \leq h(f_1,{\mathcal A}) \leq h_\mu(f_1) \leq h$, which proves the claim. Because we showed in \S\ref{sec:proof} that \ref{H1}--\ref{H4} imply uniqueness of the MME, we conclude that $\mu$ is the unique MME, and in particular is ergodic.

Now we show that $P^{-1}(\xi,\eta)$ is a single trajectory $\bar\mu$-a.e. Indeed, if $\bar\mu$ gives positive weight to the set of pairs $(\xi,\eta)$ for which $P^{-1}(\xi,\eta)$ contains more than one trajectory, then we would have $\tmu(\mathcal{E}^c)>0$, contradicting \ref{H4}. Thus $\bar\mu$-a.e.\ $(\xi,\eta)$ has the property that $P^{-1}(\xi,\eta)$ is a single trajectory, and thus \eqref{eqn:Kaim} immediately gives an invariant measure, without the need for the later averaging procedures; this measure must be $\tmu$.

Using \eqref{eqn:Kaim} we can deduce that $\mu$ is fully supported. Indeed, if $U\subset T^1\wM$ is open then $P(U) \subset \sqbd$ is open as well by the definition of the topology on $\ideal$. Thus $\bar\mu(P(U))>0$, and \eqref{eqn:Kaim} immediately gives $\tmu(U)>0$, which proves Theorem \ref{5.4.C}.

\begin{remark}
As discussed in Remark \ref{rmk:H1H2}, we could replace \ref{H3} and \ref{H4} with the assumption that the MME is unique and has a lift giving full weight to $\mathcal{E}$.
We conjecture that uniqueness immediately implies this expansivity hypothesis. Indeed, if we select for each $(\xi,\eta)\in\sqbd$ a single trajectory $V(\xi,\eta) \subset T^1\wM$ corresponding to one of the geodesics connecting $\xi$ and $\eta$, and then take $\lambda_{\xi,\eta}$ to be Lebesgue measure along this geodesic, we might expect to immediately obtain an MME (or rather its lift) by \eqref{eqn:Kaim}, and then observe that making two different choices $V_1$ and $V_2$ would give two distinct MMEs unless $P^{-1}(\xi,\eta)$ is a single trajectory $\bar\mu$-a.e.  However, it is not clear how to define $V$ in a way that is simultaneously measurable and $\Gamma$-equivariant, and so for the time being this remains a conjecture.
\end{remark}

\section{On mixing of the measure of maximal entropy}\label{sec:mixing}

In this section we follow the ideas of \cite[Theorem 2]{mB02} to prove that the MME constructed in the previous section is mixing.

\begin{theorem}\label{thm:mix}
Let \(M\) be a smooth closed Riemannian manifold without conjugate points satisfying conditions \ref{H1}--\ref{H4}. %Suppose that the geodesic flow on $T^1M$ has been shown to have a unique measure of maximal entropy $\mu$, and that $\mu$ has the following ``almost expansivity'' property: for $\mu$-a.e.\ $v\in T^1M$, every lift $\tilde v\in T^1\wM$ lies in the set $\mathcal{E}$ from \eqref{eqn:E}.
Then \(F\) is mixing with respect to the unique MME \(\mu\).
\end{theorem}

\begin{remark}
As in Remark \ref{rmk:H1H2}, we could replace \ref{H3}--\ref{H4} with the assumption that the MME is unique and that its lift gives full weight to $\mathcal{E}$.
\end{remark}

As in \cite[Theorem 2]{mB02}, the proof of Theorem \ref{thm:mix} is based on three key properties of the flow that are derived from the assumptions of Theorem \ref{thm:higher-dim}: 
\begin{itemize}
\item the product structure properties of the measure of maximal entropy that is given by the Patterson--Sullivan  construction in \S\ref{sec:PS};
\item the continuity of the cross-ratio function which is discussed in \S\ref{S:cross-ratio};
\item Lemma \ref{lem:tnk} below gives enough hyperbolicity $\mu$-a.e.\ to run a version of the Hopf argument, which uses the fact that $\mu$ is ergodic together with the previous ingredients to establish mixing.
% there is enough hyperbolicity for typical point of the MME that allows to run a version of Hopf argument; this is the content of Lemma \ref{lem:tnk}.
\end{itemize}

\subsection{The cross-ratio function}\label{S:cross-ratio}
  Most of the definitions and properties below are given in \cite{fD99} which are inspired by similar concepts in \cite{jO92}. However, since  in \cite{fD99}, the case of negative curvature is considered, for completeness, we include all the proofs. 
  
Given two distinct points \(\xi, \xi'\in\partial \wM\), \((\xi,\xi')\) denotes a geodesic, which is not necessary unique, joining \(\xi\) and \(\xi'\).  Given \(p\in \wM\) and \(\xi\in\partial\wM\), \(H_p(\xi)\) denotes the horosphere centered at \(\xi\) containing \(p\).
Observe that $H_p(\xi) = \{q\in \wM : b_p(q,\xi) = 0\} = H^s(v)$, where $v$ is the unique unit tangent vector at $p$ such that $c_v(\infty) = \xi$ (see Lemma \ref{lem:p-to-xi} and Definition \ref{def:busemann}).

When $M$ is a manifold of hyperbolic type (Definition \ref{def:hyp-type}), Pesin proved continuity of the map $\xi\mapsto H_p(\xi)$ by first observing that \cite[Lemma 1.6]{pEb72} establishes the following \emph{axiom of asymptoticity} for $\wM$: suppose that $x_n,x\in \wM$, $v_n,v\in T^1\wM$, and $t_n\to\infty$ are such that $x_n\to x$, $v_n\to v$, and let $c_n$ be the unique geodesic from $x_n$ to $c_{v_n}(t_n)$; then for any limit point $w\in T^1_x\wM$ of the sequence $\dot{c}_n(0)\in T^1_{x_n}\wM$, we have $c_v(\infty) = c_w(\infty)$ \cite[Definition 5.1 and Proposition 5.4]{jP77}. He then used the axiom of asymptoticity to prove the following continuity result \cite[Lemma 6.2]{jP77}; see also \cite[Lemma 4.11]{rG07}.

\begin{proposition}\label{prop:conti}
Let \((M,g)\) be a compact Riemannian manifold without conjugate points. If \(M\) is of hyperbolic type then for every \(p\in\wM\), the map \(\xi\to H_p(\xi)\) is continuous  where \(\{H_p(\xi), \xi\in\ideal\}\) is equipped with the compact open topology: in other words, if $\xi_n \to \xi\in \ideal$ and $K\subset \wM$ is compact, then $H_p(\xi_n) \cap K \to H_p(\xi) \cap K$ uniformly.
\end{proposition}

\begin{lemma}[Definition]\label{lem:ext}
 For \(p\in\wM,\) the length of the segment in \((\xi,\xi')\) with end points in \(H_{p}(\xi)\cap(\xi,\xi')\) and \(H_{p}(\xi')\cap(\xi,\xi')\) does not depend 
 on the choice of the geodesic \((\xi,\xi')\). In particular the Gromov product \((\xi|\xi')_{p}\) is well defined as the length of that segment, see Figure \ref{Gromov-Product}.
\end{lemma}
\begin{proof}
Let \(v,w\in T^1\wM\) associated to two different geodesics joining \(\xi\) and \(\xi'\) as in Figure \ref{Gromov-Product}. Then the proof follows 
since the flow on $\wM$ tangent to $\grad b_v$ takes horospheres to horospheres.
\end{proof}

\begin{remark}
Observe that $(\xi|\xi')_p = \beta_p(\xi,\xi')$, recall \eqref{eqn:beta-p}.
\end{remark}

\begin{figure}[htb]
\def\svgwidth{.7\textwidth}
\center{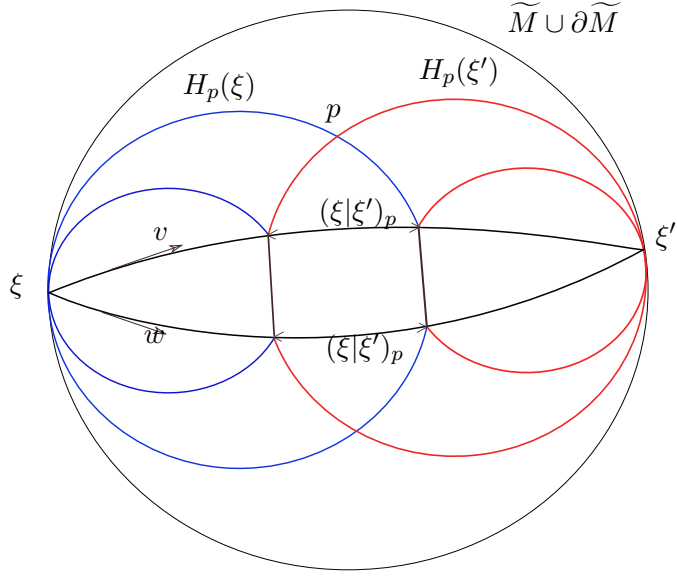}
\caption{Gromov product.}\label{Gromov-Product}
\end{figure}

%\begin{figure}[htb]
%\def\svgwidth{.5\textwidth}
%\center{\input{Hor-distance.pdf_tex}}
%\caption{Horospherical distance.}\label{Hor-distance}
%\end{figure}

Fixing a reference point $p\in \wM$,
the cross-ratio of four points \(\xi, \xi', \eta, \eta'\in\partial\wM\), with $\xi,\xi'$ distinct from $\eta,\eta'$,
is defined by 
\[
 [\xi,\xi',\eta,\eta']:=\left((\xi|\eta')_p +(\xi'|\eta)_p\right)-\left((\xi|\eta)_p +(\xi'|\eta')_p\right).
\]
We remark that from the continuity of  the map \(\xi\to H_p(\xi)\) (Proposition \ref{prop:conti}), the cross-ratio function is continuous on
\[
(\partial\wM)^{(4)} := \{ (\xi,\xi',\eta,\eta') \in (\ideal)^4 : \{\xi,\xi'\} \cap \{\eta,\eta'\} = \emptyset \}.
\]
Moreover as in \cite{mB96}, we observe that for \(q\in\wM\), 
\begin{equation}\label{eq:hor0}
(\xi|\xi')_p-(\xi|\xi')_q=b_p(q,\xi)+b_p(q,\xi')
\end{equation}

 This implies that the cross ratio does not depend on the reference point \(p\). Using Corollary \ref{lem:b-xi} and the fact that \(\Gamma\) acts on \(\wM\) by isometries, we have:
\begin{gather}\label{eq:hor1}
b_p(q,\xi)=b_p(p',\xi)+b_{p'}(q,\xi)\quad\forall p,p',q\in\wM\text{ and }\forall\xi\in\partial\wM, \\
\label{eq:hor2}
b_{\gamma(p)}(\gamma(q),\gamma(\xi))=b_p(q,\xi)\quad\forall p,q\in\wM,\ \forall \gamma\in\Gamma\text{ and }\forall\xi\in\partial\wM.
\end{gather}

Given \((\xi,\xi',\eta,\eta')\in (\ideal)^{(4)}\), we fix \(v\in(\xi,\eta)\). Let \(v_1:=(\xi',\eta)\cap H^s(v)\), \(v_2:=(\xi',\eta')\cap H^u(v_1)\), \(v_3:=(\xi,\eta')\cap H^s(v_2)\), \(v_4:=(\xi,\eta)\cap H^u(v_3)\). We have the following
\begin{figure}[htb]
\def\svgwidth{.9\textwidth}
\center{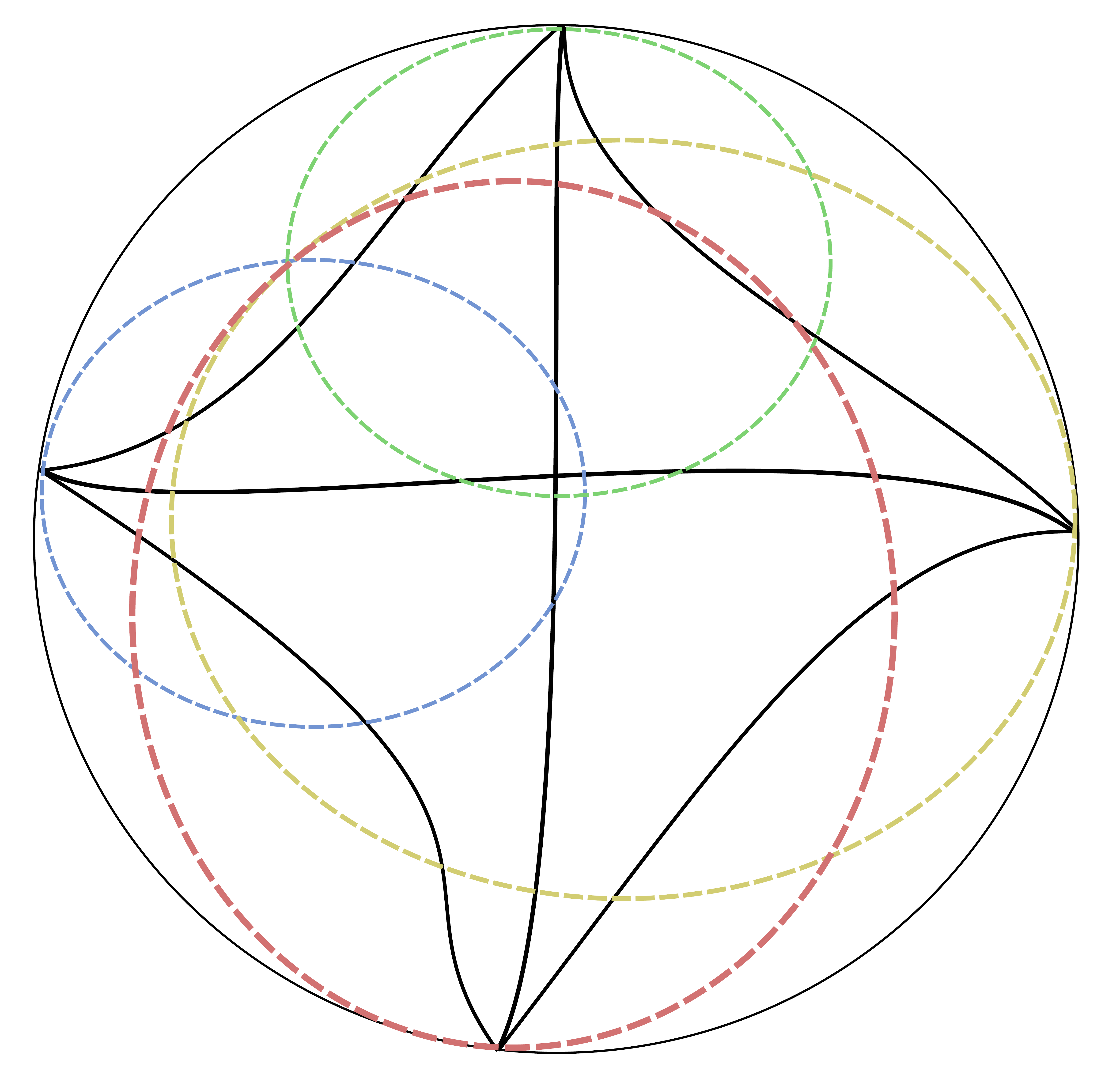}
\caption{Proof of Lemma \ref{lem:equi}.}\label{cross-ratio}
\end{figure}

\begin{lemma}\label{lem:equi}
\(v_4=f_{[\xi,\xi',\eta,\eta']}(v)\).
\end{lemma}
Lemma \ref{lem:equi} is due to Otal \cite{jO92} in the case of negative curvature and it shows the analogy between the cross-ratio function and  the temporal function in \cite[Figure 2]{cL04}. 
\begin{proof}
Let \((\xi,\xi',\eta,\eta')\in (\ideal)^{(4)}\) and \(p\in\wM\). It suffices to consider the case when $b_v(p)=0$, so that $q := \pi(v) \in H_p(\eta)$.
We refer to Figure \ref{cross-ratio} where the dotted horospheres \(H_p(\xi), H_p(\xi'), H_p(\eta), H_p(\eta')\) all pass through $p\in H_q(\eta)$ and cut out geodesic segments whose lengths are the four Gromov products involved in the definition of the cross-ratio $[\xi,\xi',\eta,\eta']$. Writing $q_i$ for the foot point of $v_i$, the figure also illustrates the fact that, from the definitions of the vectors $v_i$, we have
\[
q_1 \in H_q(\eta),\quad q_2 \in H_{q_1}(\xi'),
\quad q_3 \in H_{q_2}(\eta'),
\quad q_4 \in H_{q_3}(\xi),
\]
which in terms of Busemann functions can be written as
\begin{align*}
b_p(q,\eta) &= b_p(q_1,\eta),
&b_p(q_1,\xi') &= b_p(q_2,\xi'), \\
b_p(q_2,\eta') &= b_p(q_3,\eta'),
&b_p(q_3,\xi) &= b_p(q_4,\xi).
\end{align*}
Using these in the definition of the cross-ratio gives
\begin{align*}
[\xi,\xi',\eta,\eta'] :&= (\xi|\eta')_p + (\xi'|\eta)_p - (\xi|\eta)_p - (\xi'|\eta')_p \\
&= -b_p(q_3,\xi) - b_p(q_3,\eta') - b_p(q_1,\xi') - b_p(q_1,\eta)\\
&\quad + b_p(q_4,\xi) + b_p(q_4,\eta) + b_p(q_2,\xi') + b_p(q_2,\eta') \\
&= b_p(q_4,\eta) - b_p(q,\eta).
\end{align*}
Recalling \eqref{eq:hor1}, this last quantity is equal to $b_q(q_4,\eta)$, and since both $q,q_4$ lie on the geodesic connecting $\xi$ and $\eta$, we conclude that $v_4 = f_{b_q(q_4,\eta)}(v) = f_{[\xi,\xi',\eta,\eta']}(v)$, which finishes the proof.
\end{proof}

\subsection{Asymptotic convergence}\label{sec:asymptotic}
The aim of this section is to prove some hyperbolic estimate for almost every point with respect to the MME.

Given $v\in T^1M$, let $\tilde{v}$ be a lift of $v$ to $T^1\wM$, and let
\begin{equation}\label{eqn:Wss}
W^{ss}(\tilde v) = W^s(\tilde v) \cap \pi^{-1} H^s(\tilde v) \subset T^1\wM.
\end{equation}
Given $R>0$ and $\tilde v\in T^1M$, let
\[
W_R^{ss}(\tilde v) = \{ \tilde w \in \pi^{-1}H^s(\tilde v) : d(f_t \tilde w, f_t\tilde v) \leq R \text{ for all } t\geq 0\}.
\]
Observe that $W^{ss}(\tilde v) = \bigcup_{R>0} W^{ss}_R(\tilde v)$.
Define $W^{uu}(\tilde v)$ and $W_R^{uu}(\tilde v)$ similarly.  
Recall from \eqref{eqn:E} that
\[
\pr_*\mathcal{E} = \{v\in T^1M : W^{ss}(\tilde v) \cap W^{uu}(\tilde v) = \{\tilde v\}\}.
\]
Fix $R>0$ and consider for each $v\in T^1M$ and $t>0$ the following value:
\[
\ph_t(v) = \sup \{d(\tilde w, \tilde v) :
 %d(f_{-\tau}\tilde w, f_{-\tau}\tilde v) \leq R \text{ for all } \tau \in [0,t] \}.
f_{-t} \tilde w \in W_R^{ss}(f_{-t}\tilde v) \}
\]
\begin{lemma}
If $v\in \pr_*\mathcal{E}$, then $\ph_t(v) \searrow 0$ monotonically as $t\to\infty$.
\end{lemma}
\begin{proof}
Monotonicity follows from the fact that $f_{-t}\tilde w \in W_R^{ss}(f_{-t}\tilde v)$ implies $f_{-t'}\tilde w \in W_R^{ss}(f_{-t'}\tilde v)$ for all $t'\leq t$, so the sets in the definition of $\ph_t$ are nested decreasing as $t$ increases.  For convergence to $0$, suppose $v\in T^1M$ is such that $\ph_t(v) \not\to 0$; then there are $\delta>0$, $t_n\to \infty$, and $\tilde w_n \in f_{t_n} W_R^{ss}(f_{-t_n} \tilde v)$ such that $d(\tilde w_n,\tilde v) \geq \delta$ for all $n$.  Since $d(\tilde w_n,\tilde v)\leq R$, we can replace $\tilde w_n$ with a convergent subsequence that has $\tilde w_n \to \tilde w$, and observe that $f_{-t} \tilde w \in W_R^{ss}(f_{-t}\tilde v)$ for all $t>0$, so $\tilde w \in W_R^{ss}(\tilde v) \cap W_R^{uu}(\tilde v)$; moreover, $d(\tilde w, \tilde v) \geq \delta$, so $\tilde w \neq \tilde v$, and thus $v\notin \pr_*\mathcal{E}$.
\end{proof}

Let $\mu$ be a flow-invariant probability measure on $T^1M$ such that $\mu(\pr_*\mathcal{E})=1$.
Since each $\ph_t$ is measurable and bounded, and $\mu$ is finite, the monotone convergence theorem implies that $\ph_t\to 0$ in the $L^1$ norm.  Observing that the function $\ph_t \circ f_t$ satisfies
\[
\|\ph_t \circ f_t\|_1 = \int \ph_t \circ f_t \,d\mu = \int \ph_t \,d\mu = \|\ph_t\|_1
\]
by flow-invariance of $\mu$, we conclude that $\ph_t\circ f_t\to 0$ in $L^1$. 
Thus for any sequence $t_n\to\infty$, there is a subsequence $t_{n_k}\to\infty$ such that $\ph_{t_{n_k}} \circ f_{t_{n_k}} \to 0$ $\mu$-a.e.
Note that
\[
\ph_t(f_t v) = \sup \{ d(f_t \tilde w, f_t\tilde v) : \tilde w \in W_R^{ss}(\tilde v) \};
\]
thus we can prove the following.

\begin{lemma}\label{lem:tnk}
Let $\mu$ be a flow-invariant probability measure on $T^1M$ with $\mu(\pr_*\mathcal{E})=1$.
For every $t_n\to\infty$ there is a subsequence $t_{n_k}\to\infty$ such that
\begin{equation}\label{eqn:tnk}
d(f_{t_{n_k}} \tilde w, f_{t_{n_k}} \tilde v)\to 0
\text{ for $\mu$-a.e.\ } v\in T^1M \text{ and every } \tilde w \in W^{ss}(\tilde v).
\end{equation}
\end{lemma}
\begin{proof}
The preceding discussion shows that for each $R>0$ and every $t_n\to\infty$, there is a subsequence $t_{n_k}\to\infty$ such that \eqref{eqn:tnk} holds with $W^{ss}(\tilde v)$ replaced by $W_R^{ss}(\tilde v)$.  Applying this with $R=1,2,3,\dots$ gives a nested family of subsequences, and the usual diagonal argument gives a subsequence that works for every $R$.  Since $W^{ss}(\tilde v) = \bigcup_{R>0} W^{ss}_R(\tilde v)$, this proves the lemma.
\end{proof}

\subsection{Proof of mixing}
Now we prove that the unique MME $\mu$ for geodesic flow on $T^1M$ is mixing, using its product structure to run a version of the Hopf argument due to Babillot \cite{mB02}. %Observe that Lemma \ref{lem:tnk} applies to $\mu$ as a consequence of Proposition \ref{prop:muE}.\vc{Need to replace this with the expansivity assumption on the MME}

Suppose for a contradiction that $\mu$ is not mixing.  Then there is a continuous function $\ph$ on $T^1M$ such that $\ph\circ f_t$ does not converge weakly to $0$ in $L^2(\mu)$.  Now we need the following lemma.

\begin{lemma}[{\cite[Lemma 1]{mB02}}]
Let $(X,\mathcal{B},m,(T_t)_{t\in A})$ be a measure preserving dynamical system, where $(X,\mathcal{B})$ is a standard Borel space, $m$ a (possibly unbounded) Borel measure on $(X,\mathcal{B})$ and $(T_t)_{t\in A}$ an action of a locally compact second countable abelian group $A$ on $X$ by measure preserving transformations.  Let $\ph\in L^2(X,m)$ be a real-valued function on $X$ such that $\int \ph\,dm = 0$ if $m$ is finite.

If there exists a sequence $(t_n)$ going to infinity in $A$ such that $\ph\circ T_{t_n}$ does not converge to $0$ in the weak-$L^2$ topology, then there exist a sequence $(s_n)$ going to infinity in $A$ and a non-constant function $\psi$ in $L^2(X,m)$ such that
\[
\ph\circ T_{s_n} \to \psi
\quad\text{and}\quad
\ph\circ T_{-s_n} \to \psi
\quad\text{in the weak-$L^2$ topology}.
\]
\end{lemma}

We conclude that there is $s_n\to\infty$ and a non-constant $\psi\in L^2(T^1 M,\mu)$ such that $\ph\circ f_{\pm s_n} \to \psi$ in the weak-$L^2$ topology.  Applying Lemma \ref{lem:tnk}, we can replace $s_n$ with a subsequence such that for $\mu$-a.e.\ $v\in T^1M$, we have
\begin{equation}\label{eqn:converges}
\begin{gathered}
\lim_{n\to\infty} d(f_{s_n} \tilde w^s, f_{s_n} \tilde v) =0
\text{ for all } \tilde w^s \in W^{ss}(\tilde v), \\
\lim_{n\to\infty} d(f_{-s_n} \tilde w^u, f_{-s_n} \tilde v) =0
\text{ for all } \tilde w^u \in W^{uu}(\tilde v).
\end{gathered}
\end{equation}

\begin{lemma}[\cite{mB02}]
Let $(\ph_n)$ be a sequence that converges weakly in $L^2(X,\mathcal{B},m)$ to some function $\psi$.  Then there is a subsequence $(\ph_{n_k})$ such that the Cesaro averages
\[
A_{K^2} = \frac 1{K^2} \sum_{k=1}^{K^2} \ph_{n_k}
\]
converge almost surely to $\psi$.
\end{lemma}
\begin{proof}
In \cite{mB02} this is quoted as a consequence of the proof of the Banach--Saks theorem (see p.\ 80 of Riesz--Sz.\ Nagy 1968), which gives a subsequence such that the square of the $L^2$-norm of $A_K - \psi$ is $O(1/K)$, and then almost-sure convergence of $(A_{K^2})$ follows from Borel--Cantelli.
\end{proof}

Thus there is a set $R\subset T^1M$ such that $\mu(R)=1$ and a subsequence $s_{n_k} \to \infty$ such for every $v\in R$, the following are true:
\begin{enumerate}
\item the convergence statements in \eqref{eqn:converges} hold;
\item $\frac 1{K^2} \sum_{k=1}^{K^2} \ph (f_{\pm s_{n_k}} v) \to \psi(v)$ as $k\to\infty$.
\end{enumerate}
Let $\tilde\psi$ be a lift of $\psi$ to $T^1\wM$, and smooth $\tilde\psi$ along the flow by replacing it with $v \mapsto \int_0^\epsilon \tilde\psi(f_t v)\,dt$.  By choosing $\eps$ small enough, $\tilde\psi$ is not constant.
By continuity of $\ph$ and the two properties just listed, %and the definition of \(\psi\), 
we see that
\[
\text{if $v,w\in R$ and $\tilde w\in W^{ss}(\tilde v)$ or $\tilde w \in W^{uu}(\tilde v)$, then $\psi(v)=\psi(w)$.}
\]
There is a set $R_0$ of full $\mu$-measure such that for every $v\in R_0$, the function $t\mapsto \tilde\psi(f_t \tilde v)$ is well-defined and continuous at all real $t$; in particular, the set of periods of this function is a closed subgroup of $\mathbb{R}$.  This subgroup only depends on the geodesic: $v$ and $f_t v$ have the same subgroup for all $t\in \RR$.  By ergodicity of $\mu$, there is a single subgroup that works for $\mu$-a.e.\ $v$.  This subgroup is not all of $\RR$ since $\tilde\psi$ is not constant, and now 
the remaining parts of the proof can be carried out exactly as in \cite{mB02}:
%to complete the proof we should argue as follows (lifting this directly from \cite{mB02}):

Because $\nu_p\times \nu_p$ is a product measure,
%Thus 
there is a set $E\subset \sqbd$ of full $\nu_p\times \nu_p$ measure, a real number $a>0$, and a $\Gamma$-invariant function $\tilde\psi$ defined $\tilde \mu$-a.e.\ on $T^1\tilde M$ such that for every $(x,y)\in E$, the group of periods of $\tilde\psi$ restricted to $c_{x,y}$ is exactly $a\mathbb{Z}$.

Next step (page 69 of \cite{mB02}): for $\nu_p^4$-a.e.\ quadrilateral, the cross-ratio belongs to $a\mathbb{Z}$.  Since $\nu_p$ is fully supported on $\ideal$, \emph{every} cross-ratio of a quadrilateral belongs to $a\mathbb{Z}$.

Since the cross-ratio of $(x,x,y,y)$ is $0$, the same is true of any nearby quadrilateral, which leads to a contradiction; choose $p$ on $c_{x,y}$ and let $x',y'$ be such that the corresponding geodesic is regular, passes through $p$, and is sufficiently close to $c_{x,y}$, then we have a quadrilateral with strictly positive cross-ratio (`Fact' on page 72 of \cite{mB02}), a contradiction.

\appendix

\section{Morse}\label{sec:Morse}

\begin{proof}[Proof of Lemma \ref{lem:endpts-suffice}]
Let \(R_1>0\) and  $c_1,c_2\colon [0,T]\to \wM$ be two geodesics with
\[
d(c_1(0),c_2(0)) \leq R_1 \quad\text{and}\quad d(c_1(T),c_2(T)) \leq R_1.
\]
For $i=1,2$, let $\alpha_i \colon [0,T_i] \to \wM$ be a $g_0$-geodesic such that $\alpha_i(0) = c_i(0)$ and $\alpha_i(T_i) = c_i(T)$.
By Theorem \ref{thm:Morse}, we have $d_H(c_i,\alpha_i) \leq R_0$ for $i=1,2$, where $R_0$ depends only on $g$ and $g_0$.  Without loss of generality we assume that $T_1 \leq T_2$.  Then the triangle inequality gives
\[
d^0(\alpha_2(0),\alpha_2(T_2))
\leq d^0(\alpha_2(0),\alpha_1(0)) + d^0(\alpha_1(0),\alpha_1(T_1)) + d^0(\alpha_1(T_1),\alpha_2(T_2)).
\]
Note that $d^0(\alpha_i(0),\alpha_i(T_i)) = T_i$, and that \eqref{lem:unif-equiv} gives 
\[
d^0(\alpha_2(0),\alpha_1(0)) = d^0(c_2(0),c_1(0)) \leq A d(c_2(0),c_1(0)) \leq A R_1,
\]
with a similar bound on $d^0(\alpha_1(T_1),\alpha_2(T_2))$.  We deduce that
\begin{equation}\label{eqn:T2T1}
T_2 \leq T_1 + 2AR_1,
\end{equation}
and consequently
\[
d^0(\alpha_1(T_1),\alpha_2(T_1)) \leq d^0(\alpha_1(T_1),\alpha_2(T_2)) + |T_2 - T_1| \leq 3AR_1.
\]
Since $g_0$ is negatively curved, the function $t\mapsto d^0(\alpha_1(t),\alpha_2(t))$ is convex, and therefore achieves its maximum on $[0,T_1]$ at an endpoint; since $d^0(\alpha_1(0),\alpha_2(0)) \leq AR_1$, we conclude that
\begin{equation}\label{eqn:d-alpha}
d^0(\alpha_1(t),\alpha_2(t)) \leq 3AR_1
\quad\text{for all } t\in [0,T_1].
\end{equation}
Since $d_H(c_i,\alpha_i)\leq R_0$, for every \(t\in[0,T]\), there exist \(t_0, t'\in[0,T]\) such that \(d(c_1(t), \alpha_1(t_0)), d(c_2(t'), \alpha_2(t_0))\leq R_0\).
Using the triangle inequality via $\alpha_1(t_0),\alpha_2(t_0),c_2(t')$ together with \eqref{lem:unif-equiv} and \eqref{eqn:d-alpha}, this gives
\begin{equation}\label{eq:01}
d(c_1(t),c_2(t)) \leq 2R_0+3A^2R_1+|t-t'|.
\end{equation}
As in the proof of \eqref{eqn:T2T1}, the triangle inequality via $c_1(0)$ and $c_1(t)$ gives
\[
t' = d(c_2(0),c_2(t')) \leq R_1 + t + (2R_0 + 3A^2 R_1),
\]
and a symmetric argument gives
\[
|t' - t| \leq 2R_0 + (3A^2 + 1)R_1.
\]
This and \eqref{eqn:d-alpha} complete the proof by putting $R_2 := 4R_0 + (6A^2+1)R_1$.
\end{proof}

\bibliographystyle{amsalpha}
\bibliography{conjugate-ref}

\end{document}